\documentclass[10pt]{article}
   \usepackage{graphicx}
   \usepackage{epsfig}
   \usepackage{amssymb}
   \usepackage{amsmath}
   \usepackage{amsthm}
   \newtheorem{lemma}{Lemma}[section]
   \newtheorem{theorem}{Theorem}[section]

   \newtheorem{remark}{Remark}[section]

   {\end{description}}

   {\hfill $\bullet$ \\}

   \newcommand{\be}{\begin{equation}}
   \newcommand{\ee}{\end{equation}}

\begin{document}
    \title{A combined Lax-Wendroff/interpolation approach with finite element method for a three-dimensional system of tectonic deformation model: application to landslides in Cameroon}
   \author{Eric Ngondiep$^{\text{\,a\,b}}$ \thanks{\textbf{Email address:} ericngondiep@gmail.com}}
   \date{$^{\text{\,a\,}}$\small{Department of Mathematics and Statistics, College of Science, Imam Mohammad Ibn Saud\\ Islamic University
        (IMSIU), $90950$ Riyadh $11632,$ Saudi Arabia.}\\
     \text{\,}\\
       $^{\text{\,b\,}}$\small{Research Center for Water and Climate Change, Institute of Geological and Mining Research, 4110 Yaounde-Cameroon.}}
   \maketitle

   \textbf{Abstract.}
    This paper develops an efficient computational technique to assess the landslide responses to tectonic deformation and to predict the implications of large bedrocks landslides on the short and long-term development of the disasters. The considered equations represent a three-dimensional system of geological structure deformation subject to suitable initial and boundary conditions. The space derivatives are approximated using the finite element procedure while the approximation in time derivative is obtained using the Lax-Wendroff and interpolation techniques. The new approach is so called a modified Lax-Wendroff/interpolation method with finite element method. The modified Lax-Wendroff/interpolation scheme is employed to efficiently treat the time derivative term and to provide a suitable time step restriction for stability. Under this time step requirement, both stability and error estimates of the new algorithm are deeply analyzed using a constructed strong norm. The theory suggests that the developed computational technique is second-order accurate in time and spatial convergent with order $O(h^{3})$, where $h$ denotes the space size. A wide set of numerical examples are carried out to confirm the theoretical results and to demonstrate the utility and validity of the proposed numerical scheme. An application to landslides observed in west and center regions in Cameroon from October $2019$ to November $2024$, are discussed.\\
    \text{\,}\\

   \ \noindent {\bf Keywords:} three-dimensional system of tectonic deformation, Lax-Wendroff scheme, interpolation technique, finite element method, combined Lax-Wendroff/interpolation technique with finite element formulation, stability analysis and convergence order.\\
   \\
   {\bf AMS Subject Classification (MSC): 65M12, 65M06, 74H15}.

  \section{Introduction and motivation}\label{sec1}

   \text{\,\,\,\,\,\,\,\,\,\,} Tectonics are the processes that result in the structure and properties of the earth crust and its evolution through time. These processes deal with old cores of continents known as "cratons" and mountain-building along with the interaction of relatively rigid plates which constitute the earth's outer shell. The geometry variation of fault segments within the tectonic regime of a right lateral strike-slip yields structures which accommodate the transpressional, strike-slip and transtensional deformations. Tectonic deformations are generated in a similar manner by assuming small temperatures (less than or equal ten) and sharp difference in the orientation of strike-slip segments at the edges of an important releasing bend. Tectonic plates cover the surface of the earth, slowly move on the ground and get stuck at their edges due to friction. An earthquake releases energy in waves at the focus which travel through the earth's crust and cause the shaking which can generate several hazards or disasters including landslides, tsunamis and flooding (Figure $\ref{fig1}$, \textbf{Figure 1(a)}).

        \begin{figure}
         \begin{center}
         \begin{tabular}{c c}
         \psfig{file=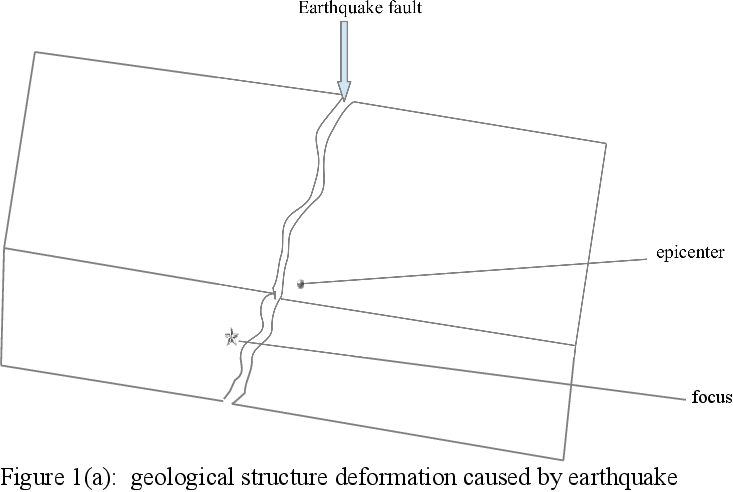,width=8cm} & \psfig{file=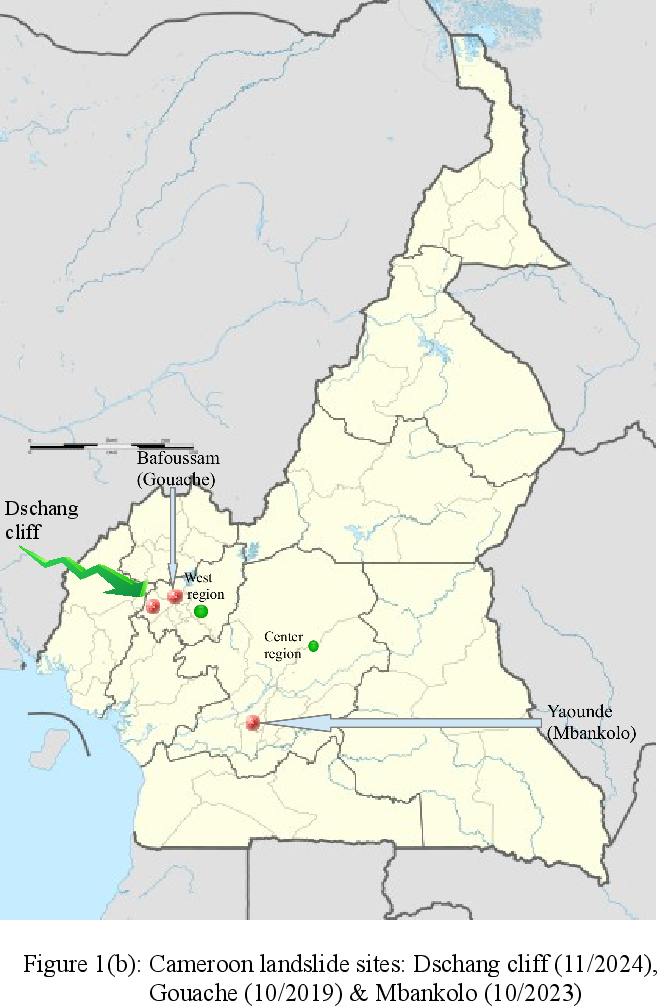,width=5.5cm}\\
         \psfig{file=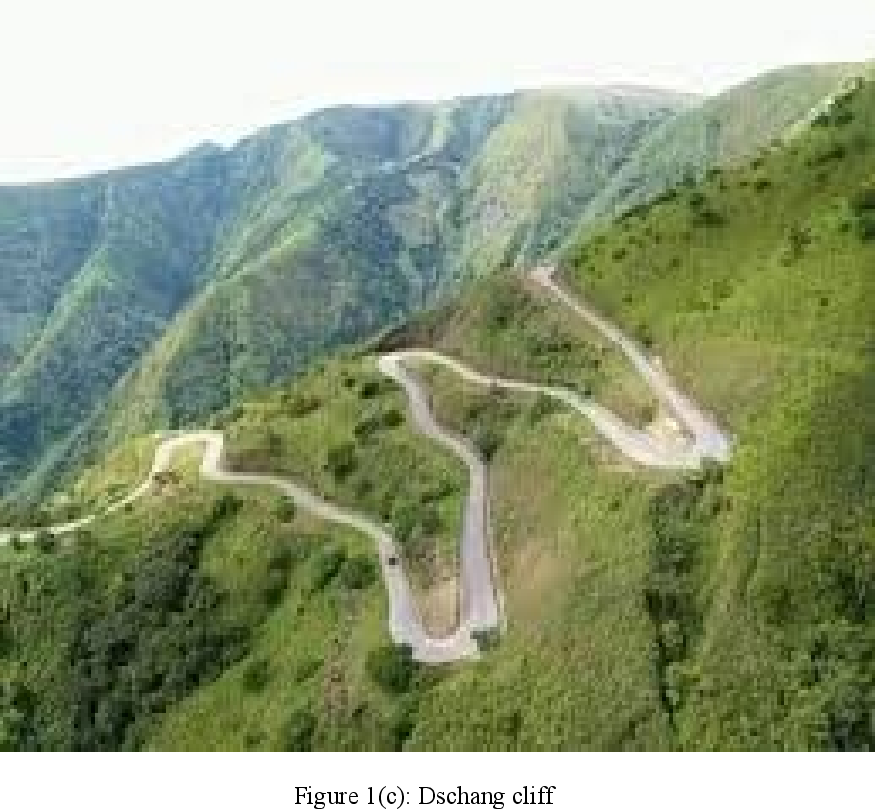,width=8cm} & \psfig{file=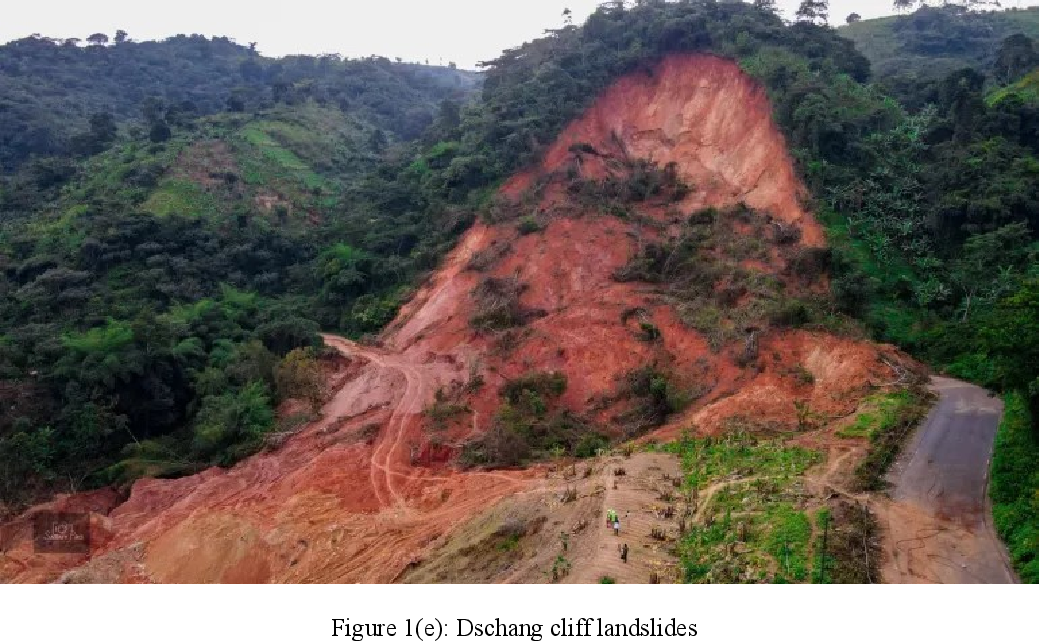,width=7.5cm}\\
         \psfig{file=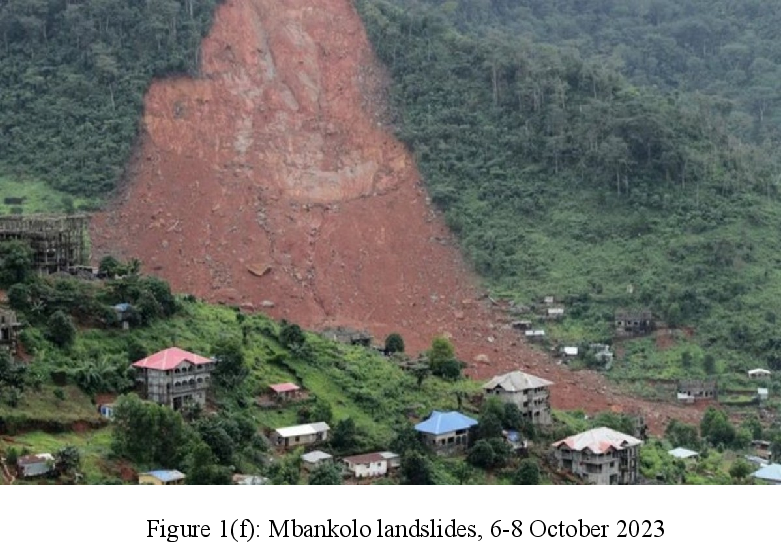,width=8cm} & \psfig{file=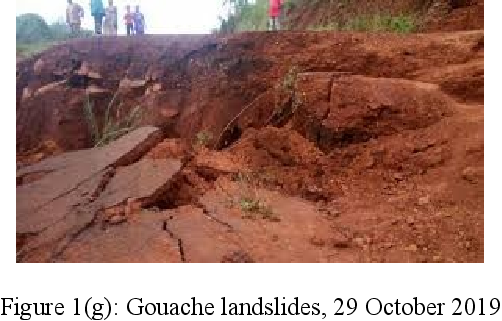,width=8cm}
         \end{tabular}
        \end{center}
        \caption{Geological structure deformation: Dschang cliff, Mbankolo and Gouache landslides in Cameroon.}
        \label{fig1}
        \end{figure}

    A landslide is the movement of a mass of debris, rock, or earth down a slope, and it is primarily caused by earthquakes, changes in water level/ground water, heavy rainfall, snowmelt, disturbance by human activities, stream erosion, and volcanic activity, whereas a tsunami is a series of massive waves caused by undersea volcanic eruptions or earthquakes. Some researchers \cite{dst} have shown that tectonic features may have an important influence in determining the location and size of landslides. Recent researches suggest that bedrock landslides have a significant role in landscape change, depending on their size and environment. However, the principal sources of geological structure deformations are earthquakes, gas extraction, and petroleum, as well as underground nuclear tests, including landslides. For more information, readers should check the works addressed in \cite{6jl,3lyl,3jl,8jl,5jl,10lyl,5lyl,4jl,7jl} and references therein.\\

   Earthquakes have become widespread in Cameroon (\textbf{Figure 1(b)}) and around the world in recent months, and roads in these areas have deteriorated \cite{echo} (\textbf{Figures 1(e)$\&$1(g)}). According to the World Health Organization's $2023$ research \cite{otherburied}, Cameroon's roads are infamously unsafe, resulting in about $3,000$ deaths annually. The causes of Cameroon landslides may include the following factors: loss of plant cover whose roots support the slope and slow down water erosion (loss of this cover is reflected in the colonization of previously vegetated areas), human activities mixed with slope action through earthworks (resulting in changes in soil cohesiveness that induce instability), site moisture content (the existence of a trickling watercourse), and sandy clay texture (effect on soil permeability and porosity). Furthermore, torrential rains caused by reliefs from climate change are the leading causes of landslides in this Central African country. The observers team and contributors provided the most recent updates on extreme weather, earthquakes, volcano eruptions, space weather, and other topics. From October $29$, $2019$, to November $5$-$7$, $2024$, heavy rains and a subsequent dam burst caused disastrous landslides and other calamities in three Cameroon cities. The avalanche impacted three buses and many road workers, resulting in over $82$ deaths and around $115$ people missing \cite{df,mn,43dead}.\\

   The death toll is expected to grow, as only four bodies have been found from the debris, with many more thought to remain buried \cite{otherburied}. Heavy rains produced a landslide on the roadway between Dschang in western Cameroon and Douala, the country's commercial metropolis. Three passenger buses and several road workers were detained, resulting in four fatalities and over $50$ missing \cite{pmutc}. Heavy rains damaged the soil, resulting in multiple calamities. Two landslides produced by excessive rainfall devastated Cameroon's west region during November $5$-$7$, $2024$, causing fatalities and damage \cite{echo}. \textbf{Figure 1(b)} shows that the Falaise community in Dschang division is experiencing road closures and highway blockages. The Dschang Cliff is located nine kilometers from Dschang town. It has an elevation of $1400m$ (latitude $5^{o}24'18"N$, longitude $10^{o}00'57"E$), a slope action of $7.6\%$, and an annual rainfall intensity of $2200mm/year$ \cite{dschtop}. The topography has north-east valleys divided by plains (\textbf{Figure 1(c)}). The first landslide occurred on the Dschang cliff road, an accident-prone stretch running through a mountainous area of the Central African nation. The second "large-scale" landslide occurred during the work and buried "three heavy machines involved in the clearing," three coaches with around $20$ seats each, five six-seater vehicles, several motorbikes, passers-by, and local residents. A second landslide near the first buried passengers, structures, and cars (\textbf{Figure 1(e)}). An unknown number of individuals were trapped under the debris, and police have reported at least four deaths and several injuries. Forecasted showers and thunderstorms may cause flooding and landslides in the region. Authorities may issue obligatory evacuation orders for flood- or landslide-prone areas. Flooding and landslides can cause disruptions to electricity and telecommunication services. Floodwaters and debris flows can make bridges, rail networks, and roadways unusable, limiting overland transport in affected areas. Ponding on road surfaces may result in hazardous driving conditions on regional routes. Authorities may temporarily restrict certain low-lying roads that become overwhelmed by flooding. In early September, a tractor-trailer transporting passengers plunged off a cliff into a ravine near the town, killing eight and wounding $62$ others, including eight children \cite{df}.\\

   The Mbankolo district, sometimes known as "Febe village," is located in Cameroon's capital, Yaounde. It is one of the city's largest quarters, with a population of almost three million. The chosen geographical coordinates are $3^{o}54'20"N$ to $11^{o}20'20"E$ for the hazard upstream at altitude $823.8m$ and $3^{o}54'36"N$ to $11^{o}29'27"E$ for the risk downstream at altitude $780m$. The topography is made up of N-E valleys divided by structural plains. It is an underdeveloped neighborhood with very uneven landscape and steep slopes. The Mfoundi river's source is located in a vast ravine in the district, which has been eroded by precipitation. At least $40$ years ago, an artificial lake was built on this site, and a broad spectrum of people resided there because of its attractiveness, which is regarded a tourist spot (\textbf{Figures 1(f)}). The steep relief of the site necessitated the construction of a support wall on the lake slope, resulting in minimal water flow. On October $6$-$8$, $2023$, heavy rains and a dam burst caused a deadly landslide in Mbankolo. Torrential rains caused a dam near an artificial lake to burst, resulting in a devastating avalanche. The lake spilled and flooded the hillside homes. The collapse of the slope destroyed around $30$ buildings consisting of wood, dried earth bricks, and metal sheets \cite{mn}. The demolition of shelters, latrines, and water sources put the affected and nearby communities at risk of water-borne diseases such as typhoid fever, malaria, cholera, and other injuries. It is difficult to provide a global appraisal of the damage due to multiple missing persons. However, the accident killed over $30$ individuals and injured at least seventeen. Red Cross workers estimated that about $700$ people, or approximately $110$ houses, were temporarily affected. The disaster left $27$ homes fully destroyed, $30$ partially destroyed, $27$ bodies discovered, and nearly $35$ people injured or missing. Approximately $40\%$ of individuals were hosted by neighbors who were not affected by the accident, while the rest sought refuge in a temporary shelter near the landslide site named Nkomkana. The affected families have lost practically everything, leaving them vulnerable. However, the Cameroon government has donated $60$ blankets, $30$ mattresses, $50$ towels, and $60$ boxes of soap to roughly $20$ survivors in public health institutions. We refer the readers to \cite{ifrc} and the references therein.\\

   Bafoussam, the headquarters of Cameroon's west region, is around $90.6km^{2}$ and has an elevation of $1323m$ above sea level. The west area is divided into eight divisions, each of which has its own subdivision, such as Bafoussam III. According to a $2005$ study by the Mifi Department of Statistics, Bafoussam had a population of approximately $348$ thousand people. Heavy landslides, like the one in Bafoussam III's Gouache area on October $29$, $2019$, are a common cause of severe disasters in the west region (\textbf{Figure 1(g)}). The Gouache catastrophe destroyed $20$ homes and killed $43$ people, leaving more than $30$ missing, $70$ injured, and several more displaced. The Cameroon government provided a USA $\$50,000$ contribution to affected persons who lost their property and needed food and shelter \cite{43dead}. However, this amount is insufficient for survivors to purchase land and rebuild their homes. According to Cameroon authorities, the damages affected at least $200$ people out of a population of around $1,000$. Torrential rains, combined with the area's steep topography, may have been the primary causes of the disaster. Furthermore, the environment is characterized by N-E valleys divided by mountains, whereas the surrounding reliefs of the disaster site are at an altitude of around $1532m$ (latitude $5^{o}24'18"N$, longitude $10^{o}00'57"E$). In $2019$, rainfall increased from January to September at the landslide site, reaching $592mm$ in September and $544mm$ in October, with the lowest recorded in December at $0.5mm$. However, an examination of rainfall data collected during the crisis period indicates that the greatest rainfall was observed on October $25$, $2019$ \cite{43dead}.\\

   In this paper we consider a tectonic deformation caused by landslides, petroleum/gas extraction and underground nuclear test and described by the three-dimensional system of elastodynamic model defined in \cite{4jl} as
    \begin{equation}\label{1}
    \left\{
      \begin{array}{ll}
        \nu w_{2t}(x,t)-(\lambda+\mu)\nabla(\nabla\cdot w(x,t))-\mu\nabla\cdot\overline{\nabla}w(x,t)=g(x,t), & \hbox{on $\Omega\times[0,\text{\,}T_{f}]$}\\
        \text{\,}\\
        \kappa(x,t)=\kappa^{0}(x)+\lambda(\nabla\cdot w(x,t))\mathcal{I}+2\mu\psi(w(x,t)), & \hbox{on $\Omega\times[0,\text{\,}T_{f}]$}
      \end{array}
    \right.
    \end{equation}
    with initial conditions
    \begin{equation}\label{2}
    w(x,0)=w_{0},\text{\,\,\,\,\,}w_{t}(x,0)=w_{1}(x),\text{\,\,\,\,\,}\kappa(x,0)=\kappa_{0}(x),\text{\,\,\,\,\,\,\,\,\,\,\,\,on\,\,\,\,}\Omega,
    \end{equation}
    and boundary condition
    \begin{equation}\label{3}
    w(x,t)=0,\text{\,\,\,\,\,\,\,\,\,\,\,\,\,\,\,\,\,\,\,\,\,\,\,\,\,\,on\,\,\,\,}\Gamma\times[0,\text{\,}T_{f}],
    \end{equation}
    where $w(x,t)$ is the displacement at time $t$ of the material particle which is located at position $x$ in the domain $\Omega$ (specifically, it allows to assess and predict damages), $\kappa$ denotes the symmetric stress tensor (symmetric matrix) that describes the internal efforts inside the material, $T_{f}$ represents the final time, $\Omega\subset\mathbb{R}^{3}$ is a bounded and connected domain, $\Gamma=\partial\Omega$ designates the boundary of $\Omega$. $w_{mt}$ means $\frac{\partial^{m}w}{\partial t^{m}}$, where $m$ is a nonnegative integer; $\nabla$, $\overline{\nabla}$, $\nabla\cdot$ and $\Delta$ denote the gradient, Jacobian, divergence and laplacian operators, respectively; $\nu$ is the elastic body density; $\lambda$ and $\mu$ are two physical parameters which may be defined as $\lambda=\frac{\alpha E}{(1+\alpha)(1-2\alpha)}$, $\mu=\frac{E}{2(1+\alpha)}$, $E$ and $\alpha$ are called elastic modulus and poisson's ratio, respectively; $\overline{\nabla}w$ denotes the jacobian matrix whereas $\nabla\cdot\nabla w$ is the vector in $\mathbb{R}^{3}$ with components $(\nabla\cdot\nabla w_{1},\nabla\cdot\nabla w_{2},\nabla\cdot\nabla w_{3})^{t}$ whenever $\nabla w=(\nabla w_{1}, \nabla w_{2},\nabla w_{3})^t$; $\psi(w)=\frac{1}{2}(\nabla w+(\nabla w)^t)$ is a symmetric matrix, so called (linearized) deformation tensor;  $\kappa_{0}$, $w_{0}$ and $w_{1}$ represent the initial conditions; $\mathcal{I}$ is the identity operator while $g(x,t)$ designates the source term (or driving force) defined as
   \begin{equation}
    g(x,t)=\left\{
             \begin{array}{ll}
               [1-(\pi g_{c}(t-t_{0}))^{2}]g_{0}(x)\exp(-(\pi g_{c}(t-t_{0}))^{2})[1,1,1]^{t}, & \hbox{for $x\in B(x_{c},r_{0})$} \\
               \text{\,}\\
               0, & \hbox{for $x\in\Omega\cup\Gamma\setminus B(x_{c},r_{0})$}
             \end{array}
           \right.
   \end{equation}
   for every $t\in[0,\text{\,}T_{f}]$, where $g_{c}$ designates a known constant (for example: rainfall intensity, slope action, site moisture content, temperature change in the medium, site sandy clay texture (soil permeability and porosity), altitude difference, etc...); $v^{t}$ indicates the transpose of a vector $v$ and $g_{0}\in H^{3}(\Omega)$ denotes a known function which vanishes on $\Gamma_{B(x_{c},r_{0})}$ along with its partial derivatives up to order $3$, where $B(x_{c},r_{0})=\{x\in\Omega,\text{\,}|x_{c}-x|<r_{0}\}$ means the opened ball centered at $x_{c}$ with radius $r_{0}$, $\Gamma_{B(x_{c},r_{0})}$ designates its boundary and $|\cdot|$ denotes the norm of the space $\mathbb{R}^{3}$. This assumption shows that $g\in[H^{2}(0,T_{f};\text{\,}H^{3}]^3$. It's worth mentioning that the opened ball $B(x_{c},r_{0})$ represents the focus where the disaster occurs. To guarantee the existence and uniqueness of a smooth solution to the considered problem $(\ref{1})$-$(\ref{3})$, we assume that the initial condition $w_{0}$ satisfies $w_{0}=0$, on $\Gamma$. The initial-boundary value problem $(\ref{1})$-$(\ref{3})$ can simulate geological structure deformation in Cameroon from October $2019$ to November $2024$. This includes the Gouache and Dschang cliff landslides ($29/10/2019$ and $5$-$7/11/2024$, respectively) in the west region, as well as the Mbankolo disaster in the center region on $8/10/2023$. In each damage site, the focus of landslides is designated by $B(x_{c},r_{0})$. Thus, the new combined Lax-Wendroff/interpolation approach with finite element method should be considered as a powerful tool to evaluate the landslides tragedy and to predict hazardous zones. However, the mathematical model of tectonic deformation described by the three-dimensional system of time-dependent partial differential equations $(\ref{1})$ lie in the set of ordinary/partial differential equations (ODEs/PDEs) \cite{1en,2en,3en,4en,5en} whose smooth solutions are too difficult to compute. Additionally, when the function $g_{0}$ and its first-order derivatives do not vanish on $\Gamma_{B(x_{c},r_{0})}$, the source term $g(x,t)$ presents some singular points. This shows that the initial-boundary value problem $(\ref{1})$-$(\ref{3})$ belongs to the class of complex models for which exact solutions are often hard to compute. For more information on similar equations, see \cite{6en,7en,8en,9en,10en} and the references therein. Efforts have been made in the literature to construct flexible, efficient, and accurate computational techniques for solving complex unsteady PDEs, such as the three-dimensional system of elastodynamic equations $(\ref{1})$. These techniques include higher-order finite difference/finite element schemes \cite{3jl,11en,18en,4jl,17en,5jl,12en,9jl,6jl,13en,7jl,14en,15en} and weak Galerkin finite element methods \cite{8jl,16en,11jl,12jl,13jl}. Although finite difference approaches remain popular, they lack "geometrical flexibility." This may be observed when dealing with varying coefficients of boundary conditions. Alternative computational approaches have been developed to address the limitations of finite element formulas when dealing with massive deformations. Because of their direct relationship to risk issues, these methodologies are commonly employed in landslide analysis, particularly when studying velocity, acceleration, and runout. In this work, we present a mixed Lax-Wendroff/interpolation methodology combined with the finite element method in a computed solution of the three-dimensional system $(\ref{1})$, subject to nonzero initial conditions $(\ref{2})$ and Dirichlet boundary condition $(\ref{3})$. To develop the desired scheme, time derivatives should be approximated with modified Lax-Wendroff and interpolation procedures, while spatial derivatives should be approximated using finite element methods. Under an appropriate time step constraint, the novel technique is stable, temporal second order accurate, and spatially third order convergent. Furthermore, the proposed approach to solve a three-dimensional elastodynamic model is faster and more efficient than a broad range of numerical systems \cite{4jl,12jl,6jl,7jl,lyl}. It is worth noting that the new approach addresses the issues identified by the numerical methods outlined above. This suggests that the constructed procedure should be considered as a powerful tool for providing useful data and related information (displacement $(w)$ and stress tensor $(\kappa)$) on some natural disasters in the world, especially in Cameroon. This would help communities to be informed about the hazard zones and allow people to avoid landslide damages. The highlights of the papers are the following.\\
   \begin{description}
     \item[(i)] A detailed description of a modified Lax-Wendroff/interpolation approach with finite element method for solving the three-dimensional system of geological structure deformation model $(\ref{1})$, subjects to initial-boundary conditions $(\ref{2})$-$(\ref{3})$.
     \item[(ii)] A deep analysis of the stability and error estimates of the proposed computational technique.
     \item[(iii)] Some numerical experiments which confirm the theoretical results and validate the accuracy, flexibility and efficiency of the new algorithm. The cases of landslides observed in some sites in Cameroon including the Dschang cliff, Mbankolo and Gouache are also considered and deeply discussed.
   \end{description}

   In the following we proceed as follows. In Section $\ref{sec2}$, we develop a modified Lax-Wendroff/interpolation scheme with finite element method in a computed solution of the three-dimensional system of tectonic deformation problem $(\ref{1})$, with initial conditions $(\ref{2})$ and boundary condition $(\ref{3})$. Section $\ref{sec3}$ deals with the stability analysis and the error estimates of the new computational technique while a wide set of numerical experiments which confirm the theoretical studies and consider the cases of landslides occurred in the west and center regions in Cameroon are presented and discussed in Section $\ref{sec4}$. In Section $\ref{sec5}$, we draw the general conclusions and indicate our future works.\\

     \section{Construction of the new approach}\label{sec2}
     This section considers a detailed description of the proposed combined explicit difference scheme with finite element method for solving a three-dimensional system of elastodynamic equation $(\ref{1})$ with appropriate initial and boundary conditions $(\ref{2})$ and $(\ref{3})$, respectively.\\

     Firstly, we defined the Sobolev spaces:
     \begin{equation*}
     \mathcal{W}=\{w=(w_{1},w_{2},w_{3})^{t}\in[W_{2}^{1}(\Omega)]^{3}:\text{\,}w|_{\Gamma}=0\};\text{\,\,}
     \mathcal{M}=\{\tau=(\tau_{ij})\in \mathcal{M}_{3}(L^{2}(\Omega)): \text{\,}\tau_{ij}=\tau_{ji},\text{\,\,}i,j=1,2,3\};
     \end{equation*}
     where $w^{t}$ designates the transpose  of the vector $w$,
     \begin{equation}\label{5}
     \mathcal{Q}=\mathcal{W}\times\mathcal{M};\text{\,\,\,} H^{4}(0,T_{f};\text{\,}H^{p})=\{v\in L^{2}(0,T_{f};\text{\,}H^{p}):\text{\,}\frac{\partial^{l}v}{\partial t^{l}}\in L^{2}(0,T_{f};\text{\,}H^{p}),\text{\,\,for\,\,}l=1,2,3,4\},
      \end{equation}
     where $\mathcal{M}_{3}(L^{2}(\Omega))$ is the space of $3\times3$-matrices with elements in $L^{2}(\Omega)$.\\

      The spaces $L^{2}(\Omega)$ and $[L^{2}(\Omega)]^{3}$ are equipped with the scalar products $\left(\cdot,\cdot\right)_{0}$ and $\left(\cdot,\cdot\right)_{\bar{0}}$, and norms $\|\cdot\|_{0}$ and $\|\cdot\|_{\bar{0}}$, respectively, whereas the spaces $W_{2}^{1}(\Omega)$, $H^{p}(\Omega)$ ($p\geq3$) and $\mathcal{W}$ are endowed with the norms $\|\cdot\|_{1}$, $\|\cdot\|_{p}$ and $\|\cdot\|_{\bar{1}}$, respectively. These inner products and norms are defined as:
      \begin{equation*}
     \left(u,v\right)_{0}=\int_{\Omega}uv d\Omega,\text{\,\,}\|u\|_{0}=\sqrt{\left(u,u\right)_{0}},\text{\,\,}\forall u,v\in L^{2}(\Omega);\text{\,\,\,}
     \left(w,z\right)_{\bar{0}}=\int_{\Omega}w^{t}z d\Omega,\text{\,\,}\|w\|_{\bar{0}}=\sqrt{\underset{i=1}{\overset{3}\sum}\|w_{i}\|_{0}^{2}},
     \end{equation*}
      \begin{equation*}
       \text{\,\,for\,\,} w=(w_{1},w_{2},w_{3})^{t},z=(z_{1},z_{2},z_{3})^{t}\in [L^{2}(\Omega)]^{3}; \text{\,\,}
       \|u\|_{1}=\sqrt{\|u\|_{0}^{2}+\|\nabla u\|_{\bar{0}}^{2}},\text{\,\,}\forall u\in W_{2}^{1}(\Omega);
     \end{equation*}
      \begin{equation}\label{6}
       \|u\|_{p}=\left(\underset{|m|=0}{\overset{p}\sum}\|D_{x}^{m}u\|_{0}^{2}\right)^{\frac{1}{2}},\text{\,\,}\forall u\in H^{p}(\Omega);\text{\,\,\,}
       \|w\|_{\bar{1}}=\sqrt{\underset{i=1}{\overset{3}\sum}\left\|w_{i}\right\|_{1}^{2}},\text{\,\,for\,\,}w=(w_{1},w_{2},w_{3})^{t}\in\mathcal{W}.
      \end{equation}
       Here, for $n=(n_{1},n_{2},n_{3})\in\mathbb{N}^{3}$ and $x=(x_{1},x_{2},x_{3})\in\mathbb{R}^{3}$, it holds: $|n|=n_{1}+n_{2}+n_{3}$, $\partial x^{|n|}=\partial x_{1}^{n_{1}}\partial x_{2}^{n_{2}}\partial x_{3}^{n_{3}}$, $D_{x}^{n}v=\frac{\partial^{|n|}v}{\partial x^{|n|}}$ and $D_{x}^{0}v:=v$. The space $\mathcal{M}$ is equipped with the scalar product $\left(\cdot,\cdot\right)_{*}$ and the norm $\|\cdot\|_{*}$, while the spaces $H^{4}(0,T_{f};\text{\,}H^{p})$ and $\mathcal{Q}=\mathcal{W}\times\mathcal{M}$ are endowed with the norms $\||\cdot|\|_{p,4}$ (also $\||\cdot|\|_{p,\infty}$) and $\||\cdot|\|_{\bar{1},*}$, respectively.
      \begin{equation*}
        \left(\tau_{1},\tau_{2}\right)_{*}=\underset{i=1}{\overset{3}\sum}\underset{j=1}{\overset{3}\sum}\int_{\Omega}(\tau_{1})_{ij}(\tau_{2})_{ij}d\Omega;\text{\,\,\,} \|\tau_{1}\|_{*}=\sqrt{\left(\tau_{1},\tau_{1}\right)_{*}},\text{\,\,for\,\,}\tau_{1}=(\tau_{1})_{ij},\tau_{2}=(\tau_{2})_{ij}\in\mathcal{M};
     \end{equation*}
      \begin{equation*}
      \||v|\|_{p,4}=\left(\underset{l=0}{\overset{4}\sum}\int_{0}^{T_{f}}\left\|\frac{\partial^{l}v(t)}{\partial t^{l}}\right\|_{p}^{2}dt\right)^{\frac{1}{2}},\text{\,\,\,} \||v|\|_{p,\infty}=\underset{0\leq t\leq T_{f}}{\max}\|v(t)\|_{p},\text{\,\,\,}\text{\,\,}v\in H^{4}(0,T_{f};\text{\,}H^{p});
      \end{equation*}
      \begin{equation}\label{7}
      \||(w,\tau)|\|_{\bar{1},*}=\sqrt{\|w\|_{\bar{1}}^{2}+\|\tau\|_{*}^{2}},\text{\,\,\,\,\,} (w,\tau)\in\mathcal{W}\times\mathcal{M}.
      \end{equation}

      We introduce the bilinear operator $A(\cdot,\cdot)$ defined as
      \begin{equation}\label{8}
      A(w,z)=(\lambda+\mu)\left(\overline{\nabla}w,\overline{\nabla}z\right)_{*}+\mu\left(\nabla\cdot w,\nabla\cdot z\right)_{0},\text{\,\,\,for\,\,\,}w,z\in\mathcal{W},
      \end{equation}
      where $\overline{\nabla}w$ denotes the Jacobian matrix of a vector function $w$.

      Let $N$ be a positive integer. Set $k=\frac{T_{f}}{N}$ be the time step and $\mathcal{T}_{k}=\{t_{n}=nk,\text{\,\,}0\leq n\leq N\}$ be a regular partition of the interval $[0,\text{\,}T_{f}]$. Let $\Pi_{h}$ be a finite element method (FEM) triangulation of the domain $\overline{\Omega}=\Omega\cup\Gamma$, which consists of tetrahedra $T$, with maximum tetrahedron diameter denoted by $"h"$. Specifically, $h$ means the step size of a spatial mesh of the computational domain $\overline{\Omega}$, whereas $\Pi_{h}$ satisfies the following assumptions: (i) the interior of any tetrahedron is nonempty; (ii) the intersection of the interior of two different tetrahedra is the empty set while the intersection of two tetrahedra is either the empty set or a common face/edge; (iii) the triangulation $\Pi_{h}$ is regular and the triangulation $\Pi_{\Gamma,h}$ induced on the boundary $\Gamma=\partial\Omega$ is quasi-uniform.\\

       Consider the finite element spaces $\mathcal{W}_{h}$ and $\mathcal{M}_{h}$ approximating the solution of the three-dimensional system of tectonic deformation $(\ref{1})$, subject to initial-boundary conditions $(\ref{2})$-$(\ref{3})$
       \begin{equation}\label{9}
      \mathcal{W}_{h}=\{w_{h}(t)\in\mathcal{W},\text{\,}w_{h}(t)|_{T}\in[\mathcal{P}_{4}(T)]^{3},\text{\,}\forall T\in\Pi_{h},\text{\,\,}\forall t\in[0,\text{\,}T_{f}]\},
      \end{equation}
      where $\mathcal{P}_{4}(T)$ is the set of polynomials defined on $T$ with degree less than or equal $4$;
      \begin{equation}\label{10}
      \mathcal{M}_{h}=\{\tau_{h}(t):=\tau(w_{h}(t))\in\mathcal{M},\text{\,}w_{h}(t)\in\mathcal{W}_{h},\text{\,\,for\,\,}0\leq t\leq T_{f}\}.
      \end{equation}

      \begin{remark}
      It's worth mentioning that the use of a computed displacement $w_{h}(t)$ obtained from the first equation in system $(\ref{1})$ together with the second equation in $(\ref{1})$, the initial conditions $(\ref{2})$ and performing direct calculations provide the corresponding symmetric stress tensor $\tau_{h}(t):=\tau(w_{h}(t))$.
      \end{remark}

      We recall the following Green formula for a vector $v\in \mathcal{W}$ and a stress tensor $\tau\in\mathcal{M}(L^{2}(\Omega))$
      \begin{equation}\label{11}
      \left(\nabla\cdot\tau,v\right)_{\bar{0}}=\int_{\Gamma}(\tau\vec{z})^{t}vd\Gamma-\left(\tau,\overline{\nabla}v\right)_{*},
      \end{equation}
      where $\vec{z}$ represents the unit outward normal vector on $\Gamma$, $w^{t}$ means the transpose of a vector $w$, $"\nabla\cdot"$ denotes the divergence operator, and $\overline{\nabla}v$ is the Jacobian matrix of the vector $v$. These operators are defined as
      \begin{equation}\label{12}
      (\nabla\cdot\tau)_{j}=\underset{i=1}{\overset{3}\sum}\frac{\partial\tau_{ij}}{\partial x_{i}},\text{\,\,\,}1\leq j\leq3\text{\,\,\,\,and\,\,\,}\overline{\nabla}v=\left(\frac{\partial v_{i}}{\partial x_{j}}\right)_{ij},\text{\,\,\,}1\leq i,j\leq3.
      \end{equation}

      For the convenience of writing, we set $w:=w(t)\in\mathcal{W}$ and $\tau:=\tau(t)\in\mathcal{M}$, for every $t\in[0,\text{\,}T_{f}]$. The application of the Taylor series expansion for the function $w$ with time step $k$ at the discrete points $t_{n-1}$ and $t_{n}$ using forward difference formula yields
      \begin{equation*}
      w^{n+1}=w^{n}+kw^{n}_{t}+\frac{k^{2}}{2}w^{n}_{2t}+\frac{k^{3}}{6}w^{n}_{3t}+\bar{O}(k^{4});\text{\,\,\,}
      w^{n}=w^{n-1}+kw^{n-1}_{t}+\frac{k^{2}}{2}w^{n-1}_{2t}+\frac{k^{3}}{6}w^{n-1}_{3t}+\bar{O}(k^{4}),
      \end{equation*}
      where $\bar{O}(k^{4})=(O(k^{4}),O(k^{4}),O(k^{4}))^{t}$. Subtracting the second equation from the first one and rearranging terms to get
      \begin{equation}\label{15}
      w^{n+1}-2w^{n}+w^{n-1}=k(w^{n}_{t}-w^{n-1}_{t})+\frac{k^{2}}{2}(w^{n}_{2t}-w^{n-1}_{2t})+\frac{k^{3}}{6}(w^{n}_{3t}-w^{n-1}_{3t})+\bar{O}(k^{4}).
      \end{equation}
      Integrating the first equation in system $(\ref{1})$ on the interval $[t_{n-1},\text{\,}t_{n}]$, multiplying the obtained equation by $\frac{1}{\nu}$ and rearranging terms, this provides
      \begin{equation}\label{16}
      w^{n}_{t}-w^{n-1}_{t}=\frac{1}{\nu}\int_{t_{n-1}}^{t_{n}}[(\lambda+\mu)\nabla(\nabla\cdot w)+\mu\nabla\cdot\overline{\nabla}w+g]dt.
      \end{equation}
      Additionally, applying the first equation in system $(\ref{1})$ at the discrete times $t_{n-1}$ and $t_{n}$ and rearranging terms yield
      \begin{equation}\label{17}
      w^{n}_{2t}=\frac{1}{\nu}[(\lambda+\mu)\nabla(\nabla\cdot w^{n})+\mu\nabla\cdot\overline{\nabla}w^{n}+g^{n}]\text{\,\,\,and\,\,\,}
      w^{n-1}_{2t}=\frac{1}{\nu}[(\lambda+\mu)\nabla(\nabla\cdot w^{n-1})+\mu\nabla\cdot\overline{\nabla}w^{n-1}+g^{n-1}].
      \end{equation}
      In addition, the use of Mean Value Theorem gives
      \begin{equation}\label{18}
      w^{n}_{3t}-w^{n-1}_{3t}=\bar{O}(k).
      \end{equation}
      Substituting equations $(\ref{16})$, $(\ref{17})$ and $(\ref{18})$ into equation $(\ref{15})$ to obtain
      \begin{equation*}
      w^{n+1}-2w^{n}+w^{n-1}=\frac{k}{\nu}\int_{t_{n-1}}^{t_{n}}[(\lambda+\mu)\nabla(\nabla\cdot w)+\mu\nabla\cdot\overline{\nabla}w+g]dt+
      \end{equation*}
      \begin{equation}\label{19}
      \frac{k^{2}}{2\nu}[(\lambda+\mu)\nabla(\nabla\cdot(w^{n}-w^{n-1})+\mu\nabla\cdot\overline{\nabla}(w^{n}-w^{n-1})+g^{n}-g^{n-1}]+\bar{O}(k^{4}).
      \end{equation}
      The approximation of $w$ at the discrete points $t_{n-1}$ and $t_{n}$ using interpolation technique gives
      \begin{equation}\label{20}
        w(t)=\frac{t-t_{n-1}}{k}w^{n}-\frac{t-t_{n}}{k}w^{n-1}+\frac{1}{2}(t-t_{n-1})(t-t_{n})w_{2t}(\epsilon_{n}(t)),
       \end{equation}
       where $\epsilon_{n}(t)$ is between the minimum and maximum of $t_{n-1}$, $t_{n}$ and $t$.\\

       Combining equations $(\ref{19})$ and $(\ref{20})$, it is easy to observe that
       \begin{equation*}
      w^{n+1}-2w^{n}+w^{n-1}=\frac{k}{\nu}\int_{t_{n-1}}^{t_{n}}\left\{(\lambda+\mu)[\frac{t-t_{n-1}}{k}\nabla(\nabla\cdot w^{n})-\frac{t-t_{n}}{k}\nabla(\nabla\cdot w^{n-1})]+\mu[\frac{t-t_{n-1}}{k}\nabla\cdot\overline{\nabla}w^{n}-\right.
      \end{equation*}
       \begin{equation*}
      \left.\frac{t-t_{n}}{k}\nabla\cdot\overline{\nabla}w^{n-1}]+g\right\}dt+\frac{k^{2}}{2\nu}[(\lambda+\mu)\nabla(\nabla\cdot(w^{n}-w^{n-1})+
      \mu\nabla\cdot\overline{\nabla}(w^{n}-w^{n-1})+g^{n}-g^{n-1}]+
      \end{equation*}
      \begin{equation}\label{21}
      \frac{k}{2\nu}\int_{t_{n-1}}^{t_{n}}(t-t_{n-1})(t-t_{n})[(\lambda+\mu)\nabla(\nabla\cdot w_{2t}(\epsilon_{n}(t)))+\mu\nabla\cdot\overline{\nabla}(w_{2t}
      (\epsilon_{n}(t)))]dt+\bar{O}(k^{4}).
      \end{equation}
      Since $w\in[H^{4}(0,T_{f};\text{\,}W_{2}^{5})]^{3}$, so $\nabla(\nabla\cdot w_{2t})$ and $\nabla\cdot\overline{\nabla}w_{2t}$ lie in $[\mathcal{C}(0,T_{f};\text{\,}W_{2}^{3})]^{3}$, where $W_{2}^{3}(\Omega)=H^{3}(\Omega)$. Further, the polynomial function $t\mapsto(t-t_{n-1})(t-t_{n})$, does not change sign on $[t_{n-1},\text{\,}t_{n}]$, thus it follows from the Weighted Mean Value Theorem \cite{bookmat333} that there is $\epsilon\in(t_{n-1},\text{\,}t_{n})$, so that
      \begin{equation*}
      \int_{t_{n-1}}^{t_{n}}(t-t_{n-1})(t-t_{n})[(\lambda+\mu)\nabla(\nabla\cdot w_{2t}(\epsilon_{n}(t)))+\mu\nabla\cdot\overline{\nabla}(w_{2t}
      (\epsilon_{n}(t)))]dt=[(\lambda+\mu)\nabla(\nabla\cdot w_{2t}(\epsilon))+
      \end{equation*}
      \begin{equation}\label{21a}
      \mu\nabla\cdot\overline{\nabla}(w_{2t}(\epsilon))]\int_{t_{n-1}}^{t_{n}}(t-t_{n-1})(t-t_{n})dt.
      \end{equation}
      But, performing straightforward computations to get $\int_{t_{n-1}}^{t_{n}}(t-t_{n-1})(t-t_{n})dt=-\frac{k^{3}}{12}$. Utilizing this, equation $(\ref{21a})$ becomes
      \begin{equation*}
      \int_{t_{n-1}}^{t_{n}}(t-t_{n-1})(t-t_{n})[(\lambda+\mu)\nabla(\nabla\cdot w_{2t}(\epsilon_{n}(t)))+\mu\nabla\cdot\overline{\nabla}(w_{2t}
      (\epsilon_{n}(t)))]dt=-\frac{k^{3}}{12}[(\lambda+\mu)\nabla(\nabla\cdot w_{2t}(\epsilon))+
      \end{equation*}
      \begin{equation}\label{22}
      \mu\nabla\cdot\overline{\nabla}(w_{2t}(\epsilon))]=\overline{O}(k^{3}).
      \end{equation}
      Substituting approximation $(\ref{22})$ into $(\ref{21})$ results in
      \begin{equation*}
      w^{n+1}-2w^{n}+w^{n-1}=\frac{k}{\nu}\left\{[(\lambda+\mu)\nabla(\nabla\cdot w^{n})+\mu\nabla\cdot\overline{\nabla}w^{n}]\int_{t_{n-1}}^{t_{n}}\frac{t-t_{n-1}}{k}dt
      -[(\lambda+\mu)\nabla(\nabla\cdot w^{n-1})+\right.
      \end{equation*}
       \begin{equation}\label{22a}
      \left.\mu\nabla\cdot\overline{\nabla}w^{n-1}]\int_{t_{n-1}}^{t_{n}}\frac{t-t_{n}}{k}dt+\int_{t_{n-1}}^{t_{n}}fdt\right\}dt+\frac{k^{2}}{2\nu}[(\lambda+\mu)
      \nabla(\nabla\cdot(w^{n}-w^{n-1})+\mu\nabla\cdot\overline{\nabla}(w^{n}-w^{n-1})+g^{n}-g^{n-1}]+\bar{O}(k^{4}).
      \end{equation}
      Performing direct computations, we obtain: $\int_{t_{n-1}}^{t_{n}}\frac{t-t_{n}}{k}dt=\frac{1}{2k}[(t-t_{n})^{2}]_{t_{n-1}}^{t_{n}}=-\frac{k}{2}$ and $\int_{t_{n-1}}^{t_{n}}\frac{t-t_{n-1}}{k}dt=\frac{1}{2k}[(t-t_{n-1})^{2}]_{t_{n-1}}^{t_{n}}=\frac{k}{2}$. This fact combined with equation $(\ref{22a})$ yield
      \begin{equation}\label{23}
      w^{n+1}-2w^{n}+w^{n-1}=\frac{k^{2}}{\nu}\left[(\lambda+\mu)\nabla(\nabla\cdot w^{n})+\mu\nabla\cdot\overline{\nabla}w^{n}\right]+
      \frac{k^{2}}{2\nu}(g^{n}-g^{n-1})+\frac{k}{\nu}\int_{t_{n-1}}^{t_{n}}gdt+\bar{O}(k^{4}).
      \end{equation}

      For every $v:=v(t)\in \mathcal{W}$, where $t\in[o,\text{\,}T_{f}]$, multiplying both sides of equation $(\ref{23})$ by $v$ and utilizing the scalar product, $\left(\cdot,\cdot\right)_{\bar{0}}$, defined in relation $(\ref{6})$, we obtain
      \begin{equation*}
      \left(w^{n+1}-2w^{n}+w^{n-1},v\right)_{\bar{0}}=\frac{k^{2}}{\nu}\left[(\lambda+\mu)\left(\nabla(\nabla\cdot w^{n}),v\right)_{\bar{0}}+
      \mu\left(\nabla\cdot\overline{\nabla}w^{n},v\right)_{\bar{0}}\right]+\frac{k^{2}}{2\nu}\left(g^{n}-g^{n-1},v\right)_{\bar{0}}+
      \end{equation*}
      \begin{equation}\label{24}
      \frac{k}{\nu}\left(\int_{t_{n-1}}^{t_{n}}gdt,v\right)_{\bar{0}}+\left(\bar{O}(k^{4}),v\right)_{\bar{0}}.
      \end{equation}

      Since $v,w\in \mathcal{W}$, so $\overline{\nabla}(w^{n}+w^{n-1})$ and $\overline{\nabla}v$ are two matrices so called Jacobian matrices and $v=0$ on $\Gamma$, thus it follows from the Green formula given by equation $(\ref{11})$ that
      \begin{equation}\label{25}
      \left(\nabla\cdot\overline{\nabla}w^{n},v\right)_{\bar{0}}=-\left(\overline{\nabla}w^{n},\overline{\nabla}v\right)_{*}.
      \end{equation}
      In addition, using the definition of the scalar products $\left(\cdot,\cdot\right)_{\bar{0}}$ and $\left(\cdot,\cdot\right)_{0}$ given in relation $(\ref{6})$ and the integration by parts to get
      \begin{equation*}
      \left(\nabla(\nabla\cdot w^{n}),v\right)_{\bar{0}}=\underset{i=1}{\overset{3}\sum}\int_{\Omega}\left(\frac{\partial}{\partial x_{i}}(\nabla\cdot w^{n})\right)v_{i}d\Omega
      =\underset{i=1}{\overset{3}\sum}\left[\int_{\Gamma}(\nabla\cdot w^{n})v_{i}\hat{z}_{i}d\Gamma-\int_{\Omega}(\nabla\cdot w^{n})\frac{\partial v_{i}}{\partial x_{i}}d\Omega\right]=
      \end{equation*}
      \begin{equation}\label{26}
      -\int_{\Omega}\nabla\cdot w^{n}\underset{i=1}{\overset{3}\sum}\frac{\partial v_{i}}{\partial x_{i}}d\Omega=-\int_{\Omega}(\nabla\cdot w^{n})(\nabla\cdot v)d\Omega=
      -\left(\nabla\cdot w^{n},\nabla\cdot v\right)_{0},
      \end{equation}
      where $\hat{z}_{i}$ denotes the ith component of the unit outward normal vector $\vec{z}$. Plugging equations $(\ref{24})$-$(\ref{26})$ and utilizing the bilinear operator $A$ defined by equation $(\ref{8})$, this gives
      \begin{equation}\label{27}
      \left(w^{n+1}-2w^{n}+w^{n-1},v\right)_{\bar{0}}=-\frac{k^{2}}{\nu}A(w^{n},v)+\frac{k^{2}}{2\nu}\left(g^{n}-g^{n-1},v\right)_{\bar{0}}+
      \frac{k}{\nu}\left(\int_{t_{n-1}}^{t_{n}}gdt,v\right)_{\bar{0}}+\left(\bar{O}(k^{4}),v\right)_{\bar{0}}.
      \end{equation}

      Furthermore, applying the second equation in system $(\ref{1})$ at the discrete time $t_{n}$, multiplying the obtained equation by $\tau\in\mathcal{M}$ and using the scalar product $\left(\cdot,\cdot\right)_{*}$ defined in relation $(\ref{7})$, we obtain
      \begin{equation}\label{27a}
      \left(\kappa^{n},\tau\right)_{*}=\left(\kappa^{0},\tau\right)_{*}+\lambda\left((\nabla\cdot w^{n})\mathcal{I},\tau\right)_{*}+2\mu\left(\psi(w^{n}),\tau\right)_{*}.
      \end{equation}

      Finally, neglecting the error term $\left(\bar{O}(k^{4}),v\right)_{\bar{0}}$ in equation $(\ref{27})$, replacing the exact solution $w(t)\in\mathcal{W}$ with the approximate one $w_{h}(t)\in\mathcal{W}_{h}$, for $t\in[0,\text{\,}T_{f}]$, rearranging terms and using equation $(\ref{27a})$ to get the new algorithm, that is, given $w_{h}^{n-1},w_{h}^{n}\in\mathcal{W}_{h}$ and $\kappa_{h}^{n}\in\mathcal{M}_{h}$, find $(w_{h}^{n+1},\kappa_{h}^{n+1})\in\mathcal{Q}_{h}=\mathcal{W}_{h}\times\mathcal{M}_{h}$, for $n=1,2,...,N-1$, so that
      \begin{equation}\label{28}
      \left(w_{h}^{n+1}-2w_{h}^{n}+w_{h}^{n-1},v\right)_{\bar{0}}=-\frac{k^{2}}{\nu}A(w_{h}^{n},v)+\frac{k^{2}}{2\nu}\left(g^{n}-g^{n-1},v\right)_{\bar{0}}+
      \frac{k}{\nu}\left(\int_{t_{n-1}}^{t_{n}}gdt,v\right)_{\bar{0}},\text{\,\,\,\,}\forall v\in \mathcal{W},
      \end{equation}
      \begin{equation}\label{28a}
      \left(\kappa_{h}^{n+1},\tau\right)_{*}=\left(\kappa^{0},\tau\right)_{*}+\lambda\left((\nabla\cdot w_{h}^{n+1})\mathcal{I},\tau\right)_{*}+
      2\mu\left(\psi(w_{h}^{n+1}),\tau\right)_{*},\text{\,\,\,\,}\forall \tau\in \mathcal{M},
      \end{equation}
      where $\mathcal{W}$ and $\mathcal{M}$ are defined in equation $(\ref{5})$. Subject to initial conditions
      \begin{equation}\label{29}
      w_{h}^{0}=w_{0},\text{\,\,\,\,}w_{h}^{1}=\widetilde{w}_{1},\text{\,\,\,\,}\kappa_{h}^{0}=\kappa^{0},\text{\,\,\,\,}\kappa_{h}^{1}=\kappa^{0}+
      \lambda(\nabla\cdot w_{h}^{1})\mathcal{I}+2\mu\psi(w_{h}^{1}),\text{\,\,\,\,on\,\,\,\,\,}\overline{\Omega}=\Omega\cup\Gamma,
      \end{equation}
      and boundary condition
      \begin{equation}\label{30}
      w_{h}^{n}=0,\text{\,\,\,\,},\kappa_{h}^{n}=\kappa^{0},\text{\,\,\,\,for\,\,\,\,}n=0,1,...,N,\text{\,\,\,\,}\text{\,\,\,\,on\,\,\,\,\,}\Gamma.
      \end{equation}
      Here $\widetilde{w}_{1}$ is the second order approximation of $w^{1}$ obtained using the second equation in relation $(\ref{2})$ and the Taylor expansion of order two. That is,
      \begin{equation}\label{31}
      \widetilde{w}_{1}=w_{0}+kw_{t}^{0}=w_{0}+kw_{1}.
      \end{equation}
      The local truncation error corresponding to this approximation is $\frac{k^{2}}{2}w_{2t}^{1}(\epsilon(t))$. Performing direct calculations, it is not hard to observe that
      \begin{equation}\label{32}
      \|w^{1}-\widetilde{w}_{1}\|_{\bar{0}}\leq \frac{k^{2}}{2}\underset{0\leq \bar{t}\leq T_{f}}{\max}\|w_{2t}(\bar{t})\|_{\bar{0}}=\||w_{2t}|\|_{\bar{0},\infty}.
      \end{equation}

      Now, we assume that the generalized sequence $\{\mathcal{W}_{h}\}_{h>0}$, of finite element subspaces approximating $\mathcal{W}$ with order $O(h)$ are used in the fluid region. Thus, the corresponding inverse estimate is defined in \cite{ch} as
      \begin{equation}\label{33}
      \left\|\frac{\partial w}{\partial x_{i}}\right\|_{\bar{0}}\leq C_{pf}h^{-1}\|w\|_{\bar{0}},\text{\,\,\,for\,\,\,}i=1,2,3,\text{\,\,\,\,\,}\forall w\in\mathcal{W}_{h},
      \end{equation}
      where $C_{pf}$ is a positive constant which is independent of the mesh space $h$ and the time step $k$. We will use the following  Poincar\'{e}-Friedrich inequality,
      \begin{equation}\label{34}
      \|u\|_{0}\leq C_{\Omega}\|\nabla u\|_{\bar{0}},\text{\,\,\,\,\,\,}\text{\,\,\,\,\,}\forall u\in W^{1}_{2}(\Omega).
      \end{equation}

      We should analyze the stability together with the error estimates of the developed combined Lax-Wendroff/interpolation approach with finite element method $(\ref{28})$-$(\ref{31})$, under the following time step restriction
      \begin{equation}\label{sr}
      \frac{k}{h}\leq C_{sr},\text{\,\,\,\,\,\,where\,\,\,\,\,}0<C_{sr}<\sqrt{2\nu}.
      \end{equation}

      \begin{remark}\label{r1}
      \begin{itemize}
        \item It's important to mention that the left hand side of inequality $(\ref{sr})$ depends on the space step $h$ so that in the usual terminology of computed solution of unsteady PDEs, the proposed explicit computational technique is stable under a suitable time step limitation. The given time step restriction $(\ref{sr})$ is more attractive and it's well known in the literature as Courant Friedrich-Lewy condition for stability of explicit numerical schemes applied to linear hyperbolic PDEs.
        \item Since the second equation in system $(\ref{1})$ indicates that the symmetric stress tensor $"\kappa"$ is not related to a differential equation and should be directly computed as a function of the displacement $w$, both stability and error estimates of the constructed algorithm $(\ref{28})$-$(\ref{31})$ must be restricted to equations satisfied by the displacement $w_{h}$.
      \end{itemize}
      \end{remark}

      To analyze the stability together with the error estimates of the proposed approach $(\ref{28})$-$(\ref{31})$, we suppose that the analytical solution satisfies the following regularity condition: $w\in[H^{4}(0,T_{f};\text{\,}H^{5})]^{3}$, which is true because the driving force $"g"$ falls in the Sobolev space $[H^{2}(0,T_{f};\text{\,}H^{3})]^{3}$. That is, there is a positive constant $\widetilde{C}$, so that
      \begin{equation}\label{35}
      \||w|\|_{\overline{5},4}\leq \widetilde{C}.
      \end{equation}

      The following Lemmas $\ref{l1}$ $\&$ $\ref{l2}$ are important in the proof of the main results (namely Theorems $\ref{t1}$ $\&$ $\ref{t2}$) of this paper.

     \begin{lemma}\label{l1}
     For every $w,v\in\mathcal{W}$, the bilinear form $A(\cdot,\cdot)$ defined by equation $(\ref{8})$ satisfies
     \begin{equation}\label{36}
     A(w,v)\leq (\lambda+4\mu)\|w\|_{\bar{1}}\|v\|_{\bar{1}}\text{\,\,\,\,\,and\,\,\,\,\,}A(w,w)\geq \frac{\lambda+\mu}{2C_{\Omega}}\min\{1,C_{\Omega}\}\|w\|_{\bar{1}}^{2}.
     \end{equation}
     \end{lemma}

   \begin{proof}
    A combination of equations $(\ref{6})$-$(\ref{8})$ along with the Cauchy-Schwarz inequality yield
    \begin{equation}\label{37}
    A(w,v)=(\lambda+\mu)\left(\overline{\nabla}w,\overline{\nabla}v\right)_{*}+\mu\left(\nabla\cdot w,\nabla\cdot v\right)_{0}\leq(\lambda+\mu)\|\overline{\nabla}w\|_{*}\|\overline{\nabla}v\|_{*}+\mu\|\nabla\cdot w\|_{0}\|\nabla\cdot v\|_{0}.
   \end{equation}
   Simple calculations provide
   \begin{equation}\label{38}
    \|\nabla\cdot w\|_{0}^{2}=\int_{\Omega}(\nabla\cdot w)^{2}d\Omega=\int_{\Omega}\left(\underset{i=1}{\overset{3}\sum}\frac{\partial w_{i}}{\partial x_{i}}\right)^{2}d\Omega\leq 3\int_{\Omega}\underset{i=1}{\overset{3}\sum}\left(\frac{\partial w_{i}}{\partial x_{i}}\right)^{2}d\Omega\leq 3\underset{i=1}{\overset{3}\sum}\|\nabla w_{i}\|_{\bar{0}}^{2}\leq 3\underset{i=1}{\overset{3}\sum}\|w_{i}\|_{1}^{2}=3\|w\|_{\bar{1}}^{2}.
    \end{equation}
    Similarly, one has
    \begin{equation}\label{39}
    \|\nabla\cdot v\|_{0}^{2}\leq 3\|v\|_{\bar{1}}^{2}.
    \end{equation}
    Additionally,
    \begin{equation}\label{40}
    \|\overline{\nabla}w\|_{*}^{2}= \int_{\Omega}\underset{i=1}{\overset{3}\sum}\underset{j=1}{\overset{3}\sum}\left(\frac{\partial w_{i}}{\partial x_{j}}\right)^{2}d\Omega=
     \underset{i=1}{\overset{3}\sum}\int_{\Omega}(\nabla w_{i})^{t}\nabla w_{i}d\Omega=\underset{i=1}{\overset{3}\sum}\|\nabla w_{i}\|_{\bar{0}}^{2}\leq \underset{i=1}{\overset{3}\sum}\|w_{i}\|_{1}^{2}=\|w\|_{\bar{1}}^{2}.
    \end{equation}
    Analogously, one easily shows that
    \begin{equation}\label{41}
    \|\overline{\nabla}v\|_{*}^{2}\leq \|v\|_{\bar{1}}^{2}.
    \end{equation}
    Taking the square root in estimates $(\ref{38})$-$(\ref{41})$ and substituting the obtained results into inequality $(\ref{37})$, we obtain
    \begin{equation*}
    A(w,v)\leq(\lambda+\mu)\|w\|_{\bar{1}}\|v\|_{\bar{1}}+3\mu\|w\|_{\bar{1}}\|v\|_{\bar{1}}=(\lambda+4\mu)\|w\|_{\bar{1}}\|v\|_{\bar{1}}.
   \end{equation*}
   This ends the proof of the first estimate in relation $(\ref{36})$.\\

   Now, utilizing the Poincar\'{e}-Friedrich inequality $(\ref{34})$, it is easy to see that
   \begin{equation}\label{42}
    \|w\|_{\bar{0}}^{2}=\underset{i=1}{\overset{3}\sum} \|w_{i}\|_{0}^{2}\leq C_{\Omega}\underset{i=1}{\overset{3}\sum}\|\nabla w_{i}\|_{\bar{0}}^{2}.
   \end{equation}
   Using this fact, it holds
   \begin{equation*}
    A(w,v)=(\lambda+\mu)\left(\overline{\nabla}w,\overline{\nabla}w\right)_{*}+\mu\left(\nabla\cdot w,\nabla\cdot w\right)_{0}=
    (\lambda+\mu)\|\overline{\nabla}w\|_{*}^{2}+\mu\|\nabla\cdot w\|_{0}^{2}\geq (\lambda+\mu)\|\overline{\nabla}w\|_{*}^{2}=(\lambda+\mu)\times
   \end{equation*}
   \begin{equation*}
    \underset{i=1}{\overset{3}\sum}\|\nabla w_{i}\|_{\bar{0}}^{2}=\frac{\lambda+\mu}{2}\underset{i=1}{\overset{3}\sum}(\|\nabla w_{i}\|_{\bar{0}}^{2}+\|\nabla w_{i}\|_{\bar{0}}^{2})\geq \frac{\lambda+\mu}{2}(C_{\Omega}^{-1}\|w_{i}\|_{\bar{0}}^{2}+\underset{i=1}{\overset{3}\sum}\|\nabla w_{i}\|_{\bar{0}}^{2})\geq \frac{\lambda+\mu}{2C_{\Omega}}\min\{1,C_{\Omega}\}\|\nabla w_{i}\|_{\bar{1}}^{2},
   \end{equation*}
   where $"\times"$ means the usual multiplication in the set of real numbers $\mathbb{R}$. The proof of Lemma $\ref{l1}$ is completed.
   \end{proof}

    \begin{remark}\label{r2}
     It's not difficult to see that the bilinear form $A(\cdot,\cdot)$ is symmetric. Thus, it follows from Lemma $\ref{l1}$ that $A(\cdot,\cdot)$ defines a scalar product on the Sobolev space $\mathcal{W}\times\mathcal{W}$. The norm $\||\cdot|\|_{A}$ associated with this scalar product is defined as
     \begin{equation}\label{43}
    \||w|\|_{A}=\sqrt{A(w,w)}=[(\lambda+\mu)\|\overline{\nabla}w\|_{*}^{2}+\mu\|\nabla\cdot w\|_{0}^{2}]^{\frac{1}{2}}.
   \end{equation}
   \end{remark}

   \begin{lemma}\label{l2}
     For every $w\in\mathcal{W}_{h}$, the following estimates hold
     \begin{equation}\label{44}
     C_{1}h^{2}\||w|\|_{A}^{2}\leq \|w\|_{\bar{0}}^{2}\leq C_{2}\||w|\|_{A}^{2},
     \end{equation}
     where $C_{1}=\frac{C_{pf}^{-2}}{18(\lambda+\mu)}$ and $C_{2}=\frac{C_{\Omega}}{\lambda+\mu}\frac{\max\{1,C_{\Omega}\}}{\min\{1,C_{\Omega}\}}$.
     \end{lemma}

   \begin{proof}
    Combining estimates $(\ref{33})$ and $(\ref{42})$, it is not hard to see that
    \begin{equation}\label{45}
    \frac{1}{3}C_{fp}^{-2}h^{2}\|\overline{\nabla}w\|_{*}^{2}\leq\|w\|_{\bar{0}}^{2}\leq C_{\Omega}\underset{i=1}{\overset{3}\sum}\|\nabla w_{i}\|_{\bar{0}}^{2}.
   \end{equation}
   But
    \begin{equation*}
    \|\nabla\cdot w\|_{0}^{2}=\int_{\Omega}\left(\underset{i=1}{\overset{3}\sum}\frac{\partial w_{i}}{\partial x_{i}}\right)^{2}d\Omega\leq 3\underset{i=1}{\overset{3}\sum}\int_{\Omega}\left(\frac{\partial w_{i}}{\partial x_{i}}\right)^{2}d\Omega\leq3\underset{i=1}{\overset{3}\sum}\left\|\frac{\partial w}{\partial x_{i}}\right\|_{\bar{0}}^{2}\leq 9C_{pf}^{2}h^{-2}\|w\|_{\bar{0}}^{2},
    \end{equation*}
    which is equivalent to
    \begin{equation*}
    \frac{1}{9}C_{pf}^{-2}h^{2}\|\nabla\cdot w\|_{0}^{2}\leq \|w\|_{\bar{0}}^{2}.
    \end{equation*}
    Plugging this inequality and estimates in $(\ref{45})$ to get
   \begin{equation*}
    \frac{C_{fp}^{-2}h^{2}}{18(\lambda+\mu)}[(\lambda+\mu)\|\overline{\nabla}w\|_{*}^{2}+\mu\|\nabla\cdot w\|_{0}^{2}]\leq\|w\|_{\bar{0}}^{2}\leq \|w\|_{\bar{0}}^{2}+C_{\Omega}\|\nabla w_{i}\|_{\bar{0}}^{2}\leq \frac{1}{2}\max\{1,C_{\Omega}\}\|w\|_{\bar{1}}^{2}.
   \end{equation*}
   Utilizing the second estimate in $(\ref{36})$ together with inequality $(\ref{43})$, this complete the proof of Lemma $\ref{l2}$.
   \end{proof}

   \section{Stability analysis and error estimates of the new algorithm}\label{sec3}
    In this Section, we analyze under the time step requirement $(\ref{sr})$, both stability and error estimates of the proposed computational technique $(\ref{28})$-$(\ref{31})$ in computed solutions of the three-dimensional system of geological structure deformation $(\ref{1})$ subjects to initial conditions $(\ref{2})$ and boundary condition $(\ref{3})$. As already indicated in the second item of Remark $\ref{r1}$, the analysis will be restricted on the equations satisfied by the displacement $w_{h}$.

    \begin{theorem} \label{t1} (Stability analysis).
     Let $w\in [H^{4}(0,T_{f};\text{\,}H^{5})]^{3}$ be the exact solution of the initial-boundary value problem $(\ref{1})$-$(\ref{3})$ and $w_{h}\in\mathcal{W}_{h}$ be the approximate solution provided by the new algorithm $(\ref{28})$-$(\ref{31})$. Under the time step limitation $(\ref{sr})$, the following estimate is satisfied
     \begin{equation*}
     \|w_{h}^{n+1}\|_{\bar{0}}^{2}+\frac{C_{1}}{\nu\gamma_{0}\gamma_{1}}\left(\||w_{h}^{n+1}|\|_{A}^{2}+\||w_{h}^{1}-w_{h}^{0}|\|_{A}^{2}\right)\leq
      \frac{C_{1}}{\nu\gamma_{0}\gamma_{1}}\left[4\nu\||w_{t}|\|_{\bar{0},\infty}^{2}+\nu k^{2}\||w_{2t}|\|_{\bar{0},\infty}^{2}\right.
     \end{equation*}
     \begin{equation*}
      \left.\||w_{0}+kw_{1}|\|_{A}^{2}+\||w_{0}|\|_{A}^{2}+\frac{9T_{f}}{\nu}\||g|\|_{\bar{0},\infty}^{2}\right]\exp(C_{1}\gamma_{0}^{-1}T_{f}),
     \end{equation*}
      for $n=0,1,...,N-1$, where $\gamma_{0}=1-\frac{C_{sr}^{2}}{2\nu}>0$, $\gamma_{1}=\min\{1,\frac{C_{1}}{2\gamma_{0}C_{2}\nu}\}$, $C_{l}$ for $l=1,2$, are given in Lemma $\ref{l2}$, while $C_{sr}$ is the positive constant defined in relation $(\ref{sr})$.
    \end{theorem}

    \begin{proof}
      It follows from equation $(\ref{28})$ that
      \begin{equation}\label{46}
      \left(w_{h}^{n+1}-2w_{h}^{n}+w_{h}^{n-1},v\right)_{\bar{0}}=-\frac{k^{2}}{\nu}A(w_{h}^{n},v)+\frac{k^{2}}{2\nu}\left(g^{n}-g^{n-1},v\right)_{\bar{0}}+
      \frac{k}{\nu}\left(\int_{t_{n-1}}^{t_{n}}gdt,v\right)_{\bar{0}},\text{\,\,\,\,}\forall v\in \mathcal{W},
      \end{equation}
      Since $g\in[H^{2}(0,T_{f};\text{\,}H^{3})]^{3}$, the function $t\mapsto g(t)$ is continuous on $[0,\text{\,}T_{f}]$. Utilizing the integral Mean Value Theorem, there exists $\epsilon_{n}\in(t_{n-1},\text{\,}t_{n})$ so that $g(\epsilon_{n})=\frac{1}{t_{n}-t_{n-1}}\int_{t_{n-1}}^{t_{n}}g(t)dt$, which is equivalent to $\int_{t_{n-1}}^{t_{n}}g(t)dt=kg(\epsilon_{n})$. Substituting this into equation $(\ref{46})$ and replacing $v$ with $w_{h}^{n+1}-w_{h}^{n-1}$, yield
      \begin{equation}\label{47}
      \left(w_{h}^{n+1}-2w_{h}^{n}+w_{h}^{n-1},w_{h}^{n+1}-w_{h}^{n-1}\right)_{\bar{0}}=-\frac{k^{2}}{\nu}A(w_{h}^{n},w_{h}^{n+1}-w_{h}^{n-1})+\frac{k^{2}}{2\nu}
      \left(g^{n}-g^{n-1}+2g(\epsilon_{n}),w_{h}^{n+1}-w_{h}^{n-1}\right)_{\bar{0}}.
      \end{equation}
      Performing straightforward computations, it holds
      \begin{equation}\label{48}
      \left(w_{h}^{n+1}-2w_{h}^{n}+w_{h}^{n-1},w_{h}^{n+1}-w_{h}^{n-1}\right)_{\bar{0}}=\|w_{h}^{n+1}-w_{h}^{n}\|_{\bar{0}}^{2}-\|w_{h}^{n}-w_{h}^{n-1}\|_{\bar{0}}^{2},
     \end{equation}
     \begin{equation}\label{49}
      -A(w_{h}^{n},w_{h}^{n+1}-w_{h}^{n-1})=\frac{1}{2}[\||w_{h}^{n+1}-w_{h}^{n}|\|_{A}^{2}-\||w_{h}^{n}-w_{h}^{n-1}|\|_{A}^{2}-\||w_{h}^{n+1}|
      \|_{A}^{2}+\||w_{h}^{n-1}|\|_{A}^{2}].
     \end{equation}

     Since $A(\cdot,\cdot)$ is a scalar product (according to Remark $\ref{r1}$), apply the Cauchy-Schwarz inequality to obtain
      \begin{equation*}
      \frac{k^{2}}{2\nu}\left(g^{n}-g^{n-1}+2g(\epsilon_{n}),w_{h}^{n+1}-w_{h}^{n-1}\right)_{\bar{0}}\leq \frac{k^{2}}{2\nu}\|g^{n}-g^{n-1}+2g(\epsilon_{n})\|_{\bar{0}}
      \|w_{h}^{n+1}-w_{h}^{n-1}\|_{\bar{0}}\leq\frac{k^{3}}{4\nu^{2}}\|g^{n}-g^{n-1}+2g(\epsilon_{n})\|_{\bar{0}}^{2}
      \end{equation*}
      \begin{equation}\label{50}
      +\frac{k}{4}\|w_{h}^{n+1}-w_{h}^{n-1}\|_{\bar{0}}^{2}\leq\frac{3k^{3}}{4\nu^{2}}(\|g^{n}\|_{\bar{0}}^{2}+\|g^{n-1}\|_{\bar{0}}^{2}+4\|g(\epsilon_{n})\|_{\bar{0}}^{2})+
      \frac{k}{2}(\|w_{h}^{n+1}-w_{h}^{n}\|_{\bar{0}}^{2}+\|w_{h}^{n}-w_{h}^{n-1}\|_{\bar{0}}^{2}).
      \end{equation}

       Substituting estimates $(\ref{48})$-$(\ref{50})$ into equation $(\ref{47})$, summing up the new estimate for $l=1,2,...,n$, observing that
     \begin{equation*}
      \underset{l=1}{\overset{n}\sum}(\|w_{h}^{l+1}-w_{h}^{l}\|_{\bar{0}}^{2}+\|w_{h}^{l}-w_{h}^{l-1}\|_{\bar{0}}^{2})\leq
      2\underset{l=0}{\overset{n}\sum}\|w_{h}^{l+1}-w_{h}^{l}\|_{\bar{0}}^{2},
     \end{equation*}
     and rearranging term, we obtain
     \begin{equation*}
      \|w_{h}^{n+1}-w_{h}^{n}\|_{\bar{0}}^{2}-\frac{k^{2}}{2\nu}\|w_{h}^{n+1}-w_{h}^{n}\|_{A}^{2}+\frac{k^{2}}{2\nu}\left(\||w_{h}^{n+1}|\|_{A}^{2}+\||w_{h}^{n}|\|_{A}^{2}+
      \||w_{h}^{1}-w_{h}^{0}|\|_{A}^{2}\right)\leq \|w_{h}^{1}-w_{h}^{0}\|_{\bar{0}}^{2}+
     \end{equation*}
     \begin{equation}\label{51}
      \frac{k^{2}}{2\nu}\left(\||w_{h}^{1}|\|_{A}^{2}+\||w_{h}^{0}|\|_{A}^{2}\right)+\frac{3k^{3}}{2\nu^{2}}\underset{l=0}{\overset{n}\sum}(\|g^{l}\|_{\bar{0}}^{2}+
      2\|g(\epsilon_{l})\|_{\bar{0}}^{2})+k\underset{l=0}{\overset{n}\sum}\|w_{h}^{l+1}-w_{h}^{l}\|_{\bar{0}}^{2}.
     \end{equation}

     Using the time step requirement $(\ref{sr})$, it holds: $k\|w_{h}^{n+1}-w_{h}^{n}\|_{A}\leq C_{sr}h\|w_{h}^{n+1}-w_{h}^{n}\|_{A}$, which implies
     $-\frac{C_{sr}^{2}h^{2}}{2\nu}\|w_{h}^{n+1}-w_{h}^{n}\|_{A}^{2}\leq -\frac{k^{2}}{2\nu}\|w_{h}^{n+1}-w_{h}^{n}\|_{A}^{2}$. This fact combined estimate $(\ref{44})$ imply
     \begin{equation}\label{52}
     -\frac{C_{sr}^{2}}{2\nu}\|w_{h}^{n+1}-w_{h}^{n}\|_{\bar{0}}^{2}\leq -\frac{C_{1}C_{sr}^{2}h^{2}}{2\nu}\|w_{h}^{n+1}-w_{h}^{n}\|_{A}^{2}\leq -\frac{C_{1}k^{2}}{2\nu}\|w_{h}^{n+1}-w_{h}^{n}\|_{A}^{2},
     \end{equation}
     where $C_{1}$ is the positive constant given in Lemma $\ref{l2}$. Multiplying both sides of estimate $(\ref{51})$ by $C_{1}$ and using inequality $(\ref{52})$, this provides
     \begin{equation*}
     \left(1-\frac{C_{sr}^{2}}{2\nu}\right)\|w_{h}^{n+1}-w_{h}^{n}\|_{\bar{0}}^{2}+\frac{C_{1}k^{2}}{2\nu}\left(\||w_{h}^{n+1}|\|_{A}^{2}+\||w_{h}^{n}|\|_{A}^{2}+
      \||w_{h}^{1}-w_{h}^{0}|\|_{A}^{2}\right)\leq C_{1}\|w_{h}^{1}-w_{h}^{0}\|_{\bar{0}}^{2}+
     \end{equation*}
     \begin{equation}\label{53}
      \frac{C_{1}k^{2}}{2\nu}\left(\||w_{h}^{1}|\|_{A}^{2}+\||w_{h}^{0}|\|_{A}^{2}\right)+\frac{3C_{1}k^{3}}{2\nu^{2}}\underset{l=0}{\overset{n}\sum}(\|g^{l}\|_{\bar{0}}^{2}+
      2\|g(\epsilon_{l})\|_{\bar{0}}^{2})+C_{1}k\underset{l=0}{\overset{n}\sum}\|w_{h}^{l+1}-w_{h}^{l}\|_{\bar{0}}^{2}.
     \end{equation}
      Since $k=\frac{T_{f}}{N}$ and $0<C_{sr}<\sqrt{2\nu}$, so $k\underset{l=0}{\overset{n}\sum}(\|g^{l}\|_{\bar{0}}^{2}+2\|g(\epsilon_{l})\|_{\bar{0}}^{2})\leq (n+1)k\underset{0\leq l\leq n}{\max}(\|g^{l}\|_{\bar{0}}^{2}+2\|g(\epsilon_{l})\|_{\bar{0}}^{2})\leq T_{f}\underset{0\leq l\leq N}{\max}(\|g^{l}\|_{\bar{0}}^{2}+2\|g(\epsilon_{l})\|_{\bar{0}}^{2})=3T_{f}\||g|\|_{\bar{0},\infty}^{2}$ and $1-\frac{C_{sr}^{2}}{2\nu}>0$. Setting $\gamma_{0}=1-\frac{C_{sr}^{2}}{2\nu}$, for small values of the time step $k$, multiplying both sides of estimate $(\ref{53})$ by $\gamma_{0}^{-1}$, applying the discrete Gronwall inequality and utilizing estimate $(n+1)k\leq T_{f}$, for $n=1,2,...,N-1$, this gives
     \begin{equation*}
     \|w_{h}^{n+1}-w_{h}^{n}\|_{\bar{0}}^{2}+\frac{C_{1}k^{2}}{2\nu\gamma_{0}}\left(\||w_{h}^{n+1}|\|_{A}^{2}+\||w_{h}^{n}|\|_{A}^{2}+
      \||w_{h}^{1}-w_{h}^{0}|\|_{A}^{2}\right)\leq \frac{C_{1}}{2\nu\gamma_{0}}\left[2\nu\|w_{h}^{1}-w_{h}^{0}\|_{\bar{0}}^{2}+\right.
     \end{equation*}
     \begin{equation}\label{54}
      \left.k^{2}\left(\||w_{h}^{1}|\|_{A}^{2}+\||w_{h}^{0}|\|_{A}^{2}\right)+\frac{9T_{f}k^{2}}{\nu}\||g|\|_{\bar{0},\infty}^{2}\right]
      \exp(C_{1}\gamma_{0}^{-1}T_{f}).
     \end{equation}

      But it follows from the triangular inequality that $\|w_{h}^{n+1}\|_{\bar{0}}^{2}\leq(\|w_{h}^{n+1}-w_{h}^{n}\|_{\bar{0}}+\|w_{h}^{n}\|_{\bar{0}})^{2}\leq 2(\|w_{h}^{n+1}-w_{h}^{n}\|_{\bar{0}}^{2}+ \|w_{h}^{n}\|_{\bar{0}}^{2})$. Utilizing this fact along with the second estimate in relation $(\ref{44})$, it is not hard to observe that
      \begin{equation*}
     k^{2}\min\left\{1,\frac{C_{1}}{2C_{2}\nu\gamma_{0}}\right\}\|w_{h}^{n+1}\|_{\bar{0}}^{2}\leq2k^{2}\min\left\{1,\frac{C_{1}}{2C_{2}\nu\gamma_{0}}\right\}
     (\|w_{h}^{n+1}-w_{h}^{n}\|_{\bar{0}}^{2}+\|w_{h}^{n}\|_{\bar{0}}^{2})\leq 2\|w_{h}^{n+1}-w_{h}^{n}\|_{\bar{0}}^{2}+\frac{C_{1}k^{2}}{\nu\gamma_{0}}\|w_{h}^{n}\|_{A}^{2}.
     \end{equation*}

      Multiplying estimate $(\ref{54})$ by $2$ and utilizing the above inequalities yield
      \begin{equation*}
     k^{2}\min\left\{1,\frac{C_{1}}{2C_{2}\nu\gamma_{0}}\right\}\|w_{h}^{n+1}\|_{\bar{0}}^{2}+\frac{C_{1}k^{2}}{\nu\gamma_{0}}\left(\||w_{h}^{n+1}|\|_{A}^{2}
     +\||w_{h}^{1}-w_{h}^{0}|\|_{A}^{2}\right)\leq \frac{C_{1}}{\nu\gamma_{0}}\left[2\nu\|w_{h}^{1}-w_{h}^{0}\|_{\bar{0}}^{2}+\right.
     \end{equation*}
     \begin{equation}\label{55}
      \left.k^{2}\left(\||w_{h}^{1}|\|_{A}^{2}+\||w_{h}^{0}|\|_{A}^{2}\right)+\frac{9T_{f}k^{2}}{\nu}\||g|\|_{\bar{0},\infty}^{2}\right]
      \exp(C_{1}\gamma_{0}^{-1}T_{f}).
     \end{equation}

      It follows from the initial conditions $(\ref{2})$ and $(\ref{29})$ together with estimate $(\ref{32})$ and the Taylor expansion that
      \begin{equation}\label{56}
      \|w_{h}^{1}-w_{h}^{0}\|_{\bar{0}}^{2}\leq 2(\|w_{h}^{1}-w^{1}\|_{\bar{0}}^{2}+\|w^{1}-w_{0}\|_{\bar{0}}^{2})\leq \frac{1}{2}k^{4}\||w_{2t}|\|_{\bar{0},\infty}^{2}+2k^{2}\||w_{t}|\|_{\bar{0},\infty}^{2},
      \end{equation}
      since $w^{1}=w_{0}+kw_{t}^{0}+\frac{k^{2}}{2}w_{2t}(\epsilon_{1})$, where $w_{t}^{0}=w_{1}$ and $0<\epsilon_{1}<t_{1}$. Setting $\gamma_{1}=\min\left\{1,\frac{C_{1}}{2C_{2}\nu\gamma_{0}}\right\}$, substituting estimate $(\ref{56})$ into inequality $(\ref{55})$ and multiplying the new estimate by $\gamma_{1}^{-1}k^{-2}$, to complete the proof of Theorem $\ref{t1}$.
      \end{proof}

      In the following we should establish the convergence order of the proposed modified Lax-Wendroff/interpolation technique with finite element method $(\ref{28})$-$(\ref{31})$ for solving a three-dimensional system of tectonic deformation $(\ref{1})$ subjects to initial and boundary conditions $(\ref{2})$ and $(\ref{3})$, respectively. We assume that the finite element subspaces $\mathcal{W}_{h}$ satisfy the usual approximation properties of piecewise polynomials of degree $3$ and $4$, that is:
      \begin{equation}\label{57}
      \underset{w_{h}\in\mathcal{W}_{h}}{\inf}\|w-w_{h}\|_{\bar{0}}\leq C_{3}h^{4}\|w\|_{\overline{4}},\text{\,\,\,\,\,\,\,\,\,\,}\text{\,\,\,\,\,\,\,\,\,}\forall w\in\mathcal{W},
     \end{equation}
     where $\|\cdot\|_{\overline{4}}$ denotes the norm defined as $\|w\|_{\overline{4}}=\left(\underset{i=1}{\overset{3}\sum}\|w\|_{4}^{2}\right)^{\frac{1}{2}}$, for any $w=(w_{1},w_{2},w_{3})\in[W_{2}^{4}(\Omega)]^{3}$.

     \begin{theorem} \label{t2} (Error estimates).
      Consider $w\in[H^{4}(0,T_{f};\text{\,}W_{2}^{5})]^{3}$, be the analytical solution of the initial-boundary value problem $(\ref{1})$-$(\ref{3})$ and let $w_{h}$ be the computed one provided by the developed computational technique $(\ref{28})$-$(\ref{31})$. Set $e_{h}^{n}=w_{h}^{n}-w^{n}$ be the error term at the discrete time $t_{n}$. Under the time step restriction $(\ref{sr})$, it holds
     \begin{equation*}
     \|e_{h}^{n+1}\|_{\bar{0}}^{2}+\frac{C_{1}}{\nu\gamma_{0}\gamma_{1}}\||e_{h}^{n+1}|\|_{A}^{2}\leq \widehat{C}(k^{2}+h^{3})^{2},
     \end{equation*}
      for $n=0,1,...,N-1$, where $\gamma_{0}$, $\gamma_{1}$, $C_{1}$ and $C_{2}$ are the constants given in Theorem $\ref{t1}$, and $\widehat{C}$ is a positive constant independent of the grid space $h$ and the time step $k$.
     \end{theorem}

     \begin{proof}
      Subtracting equation $(\ref{27})$ from approximation $(\ref{28})$, rearranging terms and taking $v=e_{h}^{n+1}-e_{h}^{n-1}\in\mathcal{W}$ to obtain
      \begin{equation}\label{58}
      \left(e_{h}^{n+1}-2e_{h}^{n}+e_{h}^{n-1},e_{h}^{n+1}-e_{h}^{n-1}\right)_{\bar{0}}=-\frac{k^{2}}{\nu}A(e_{h}^{n},e_{h}^{n+1}-e_{h}^{n-1})+
      \left(\bar{O}(k^{4}),e_{h}^{n+1}-e_{h}^{n-1}\right)_{\bar{0}}.
     \end{equation}
      It is not difficult to show that
     \begin{equation}\label{59}
      \left(e_{h}^{n+1}-2e_{h}^{n}+e_{h}^{n-1},e_{h}^{n+1}-e_{h}^{n-1}\right)_{\bar{0}}=\|e_{h}^{n+1}-e_{h}^{n}\|_{\bar{0}}^{2}-\|e_{h}^{n}-e_{h}^{n-1}\|_{\bar{0}}^{2},
     \end{equation}
     and
     \begin{equation}\label{60}
     -A(e_{h}^{n},e_{h}^{n+1}-e_{h}^{n-1})]=\frac{1}{2}(\||e_{h}^{n+1}-e_{h}^{n}|\|_{A}^{2}-\||e_{h}^{n}-e_{h}^{n-1}|\|_{A}^{2})-\frac{1}{2}(\||e_{h}^{n+1}|\|_{A}^{2}-
     \||e_{h}^{n}|\|_{A}^{2})-\frac{1}{2}(\||e_{h}^{n}|\|_{A}^{2}-\||e_{h}^{n-1}|\|_{A}^{2}).
     \end{equation}

     But there exists a positive constant $C_{4}$ independent of the time step $k$ and mesh grid $h$ so that $\||\bar{O}(k^{4})|\|_{\bar{0}}\leq C_{4}k^{4}$. Applying the Cauchy-Schwarz and triangular inequalities, direct calculations yield
     \begin{equation}\label{61}
     \left(\bar{O}(k^{4}),e_{h}^{n+1}-e_{h}^{n-1}\right)_{\bar{0}}\leq C_{4}^{2}k^{7}+\frac{k}{4}\|e_{h}^{n+1}-e_{h}^{n-1}\|_{\bar{0}}^{2}\leq C_{4}^{2}k^{7}+\frac{k}{2}(\|e_{h}^{n+1}-e_{h}^{n}\|_{\bar{0}}^{2}+\|e_{h}^{n}-e_{h}^{n-1}\|_{\bar{0}}^{2}).
     \end{equation}

      Substituting equations $(\ref{59})$-$(\ref{60})$ and inequality $(\ref{61})$ into equation $(\ref{58})$, results in
      \begin{equation*}
      \|e_{h}^{n+1}-e_{h}^{n}\|_{\bar{0}}^{2}-\|e_{h}^{n}-e_{h}^{n-1}\|_{\bar{0}}^{2}\leq \frac{k^{2}}{2\nu}\left[(\||e_{h}^{n+1}-e_{h}^{n}|\|_{A}^{2}-
      \||e_{h}^{n}-e_{h}^{n-1}|\|_{A}^{2})-(\||e_{h}^{n+1}|\|_{A}^{2}-\||e_{h}^{n}|\|_{A}^{2})-\right.
     \end{equation*}
     \begin{equation*}
      \left.(\||e_{h}^{n}|\|_{A}^{2}-\||e_{h}^{n-1}|\|_{A}^{2})\right]+C_{4}^{2}k^{7}+\frac{k}{2}(\|e_{h}^{n+1}-e_{h}^{n}\|_{\bar{0}}^{2}+
      \|e_{h}^{n}-e_{h}^{n-1}\|_{\bar{0}}^{2}).
     \end{equation*}

      Summing up for $l=1,2,...,n$, observing that
     \begin{equation*}
      \underset{p=1}{\overset{n}\sum}(\|e_{h}^{l+1}-e_{h}^{l}\|_{\bar{0}}^{2}+\|e_{h}^{l}-e_{h}^{l-1}\|_{\bar{0}}^{2})\leq 2\underset{l=0}{\overset{n}\sum}
      \|e_{h}^{l+1}-e_{h}^{l}\|_{\bar{0}}^{2},
     \end{equation*}
     and rearranging term, this provides
     \begin{equation*}
      \|e_{h}^{n+1}-e_{h}^{n}\|_{\bar{0}}^{2}-\frac{k^{2}}{2\nu}\||e_{h}^{n+1}-e_{h}^{n}|\|_{A}^{2}+\frac{k^{2}}{2\nu}\left(\||e_{h}^{n+1}|\|_{A}^{2}+\||e_{h}^{n}|\|_{A}^{2}+
      \||e_{h}^{1}-e_{h}^{0}|\|_{A}^{2}\right)\leq \|e_{h}^{1}-e_{h}^{0}\|_{\bar{0}}^{2}+
     \end{equation*}
     \begin{equation}\label{62}
      \frac{k^{2}}{2\nu}\left(\||e_{h}^{1}|\|_{A}^{2}+\||e_{h}^{0}|\|_{A}^{2}\right)+k\underset{l=0}{\overset{n}\sum}\|e_{h}^{l+1}-e_{h}^{l}\|_{\bar{0}}^{2}+C_{4}^{2}T_{f}k^{6}.
     \end{equation}

     Using the time step limitation $(\ref{sr})$, one easily proves as in estimate $(\ref{52})$ that
     \begin{equation}\label{63}
     -\frac{C_{sr}^{2}}{2\nu}\|e_{h}^{n+1}-e_{h}^{n}\|_{\bar{0}}^{2}\leq -\frac{C_{1}C_{sr}^{2}h^{2}}{2\nu}\||e_{h}^{n+1}-e_{h}^{n}|\|_{A}^{2}\leq -\frac{C_{1}k^{2}}{2\nu}\||e_{h}^{n+1}-e_{h}^{n}|\|_{A}^{2},
     \end{equation}
     where $C_{1}=\frac{C_{pf}^{-2}}{18(\lambda+\mu)}$. Multiplying inequality $(\ref{62})$ by $C_{1}$ and utilizing estimate $(\ref{63})$ give
     \begin{equation*}
     \|e_{h}^{n+1}-e_{h}^{n}\|_{\bar{0}}^{2}+\frac{C_{1}k^{2}}{2\gamma_{0}\nu}\left(\||e_{h}^{n+1}|\|_{A}^{2}+\||e_{h}^{n}|\|_{A}^{2}+
      \||e_{h}^{1}-e_{h}^{0}|\|_{A}^{2}\right)\leq C_{1}\gamma_{0}^{-1}\|e_{h}^{1}-e_{h}^{0}\|_{\bar{0}}^{2}+
     \end{equation*}
     \begin{equation*}
      \frac{C_{1}k^{2}}{2\gamma_{0}\nu}\left(\||e_{h}^{1}|\|_{A}^{2}+\||e_{h}^{0}|\|_{A}^{2}\right)+C_{1}\gamma_{0}^{-1}k\underset{l=0}{\overset{n}\sum}
      \|e_{h}^{l+1}-e_{h}^{l}\|_{\bar{0}}^{2}+C_{1}C_{4}^{2}\gamma_{0}^{-1}T_{f}k^{6},
     \end{equation*}
     where $\gamma_{0}=1-\frac{C_{sr}^{2}}{2\nu}>0$, since $0<C_{sr}<\sqrt{2\nu}$. For small values of $k$ satisfying the time step requirement $(\ref{sr})$, the application of the discrete Gronwall inequality results in
     \begin{equation*}
     \|e_{h}^{n+1}-e_{h}^{n}\|_{\bar{0}}^{2}+\frac{C_{1}k^{2}}{2\gamma_{0}\nu}\left(\||e_{h}^{n+1}|\|_{A}^{2}+\||e_{h}^{n}|\|_{A}^{2}+
      \||e_{h}^{1}-e_{h}^{0}|\|_{A}^{2}\right)\leq \frac{C_{1}}{2\gamma_{0}\nu}\left[2\nu\|e_{h}^{1}-e_{h}^{0}\|_{\bar{0}}^{2}+\right.
     \end{equation*}
     \begin{equation}\label{64}
      \left.k^{2}\left(\||e_{h}^{1}|\|_{A}^{2}+\||e_{h}^{0}|\|_{A}^{2}\right)+2\nu C_{4}^{2}T_{f}k^{6}\right]\exp(C_{1}\gamma_{0}^{-1}T_{f}).
     \end{equation}

      As in the proof of Theorem $\ref{t1}$, it is not hard to show that
     \begin{equation*}
      k^{2}\min\left\{1,\frac{C_{1}}{2C_{2}\gamma_{0}\nu}\right\}\|e_{h}^{n+1}\|_{\bar{0}}^{2}\leq 2\|e_{h}^{n+1}-e_{h}^{n}\|_{\bar{0}}^{2}+
      \frac{C_{1}k^{2}}{\gamma_{0}\nu}\|e_{h}^{n}\|_{A}^{2}.
     \end{equation*}

     Utilizing this fact and performing direct computations, estimate $(\ref{64})$ implies
      \begin{equation*}
     \|e_{h}^{n+1}\|_{\bar{0}}^{2}+\frac{C_{1}}{\gamma_{0}\gamma_{1}\nu}\left(\||e_{h}^{n+1}|\|_{A}^{2}
     +\||e_{h}^{1}-e_{h}^{0}|\|_{A}^{2}\right)\leq \frac{C_{1}k^{-2}}{\gamma_{0}\gamma_{1}\nu}\left[2\nu\|e_{h}^{1}-e_{h}^{0}\|_{\bar{0}}^{2}+\right.
     \end{equation*}
     \begin{equation*}
      \left.k^{2}\left(\||e_{h}^{1}|\|_{A}^{2}+\||e_{h}^{0}|\|_{A}^{2}\right)+2\nu C_{4}^{2}T_{f}k^{6}\right]\exp(C_{1}\gamma_{0}^{-1}T_{f}),
     \end{equation*}
      where $\gamma_{1}=\min\left\{1,\frac{C_{1}}{2C_{2}\gamma_{0}\nu}\right\}$. But, it follows from the initial conditions $(\ref{29})$ that $w^{0}=w_{h}^{0}=w_{0}$, so $e_{h}^{0}=\vec{0}$, where $\vec{0}$ means the zero vector of $\mathcal{W}$, $\||e_{h}^{1}-e_{h}^{0}|\|_{A}=\||e_{h}^{1}|\|_{A}$ and $\|e_{h}^{1}-e_{h}^{0}\|_{\bar{0}}=\|e_{h}^{1}\|_{\bar{0}}$. Adding the term $-\frac{C_{1}}{\gamma_{0}\gamma_{1}\nu}\||e_{h}^{1}|\|_{A}^{2}$ in both sides of this estimate to obtain
      \begin{equation}\label{65}
     \|w_{h}^{n+1}-w^{n+1}\|_{\bar{0}}^{2}+\frac{C_{1}}{\gamma_{0}\gamma_{1}\nu}\||w_{h}^{n+1}-w^{n+1}|\|_{A}^{2}\leq \frac{2C_{1}}{\gamma_{0}\gamma_{1}}
     \left[k^{-2}\|w_{h}^{1}-w^{1}\|_{\bar{0}}^{2}+C_{4}^{2}T_{f}k^{4}\right]\exp(C_{1}\gamma_{0}^{-1}T_{f}).
     \end{equation}

     Let $\widehat{w}\in \mathcal{W}_{h}$ be an arbitrary vector. So $\|w_{h}^{1}-w^{1}\|_{\bar{0}}^{2}\leq 2(\|w_{h}^{1}-\widehat{w}\|_{\bar{0}}^{2}+
     \|\widehat{w}-w^{1})\|_{\bar{0}}^{2})$. This fact combined with estimate $(\ref{65})$ imply
     \begin{equation*}
     \|w_{h}^{n+1}-w^{n+1}\|_{\bar{0}}^{2}+\frac{C_{1}}{\gamma_{0}\gamma_{1}\nu}\||w_{h}^{n+1}-w^{n+1}|\|_{A}^{2}\leq \frac{4C_{1}}{\gamma_{0}\gamma_{1}}
     \left[k^{-2}(\|w_{h}^{1}-\widehat{w}\|_{\bar{0}}^{2}+\|\widehat{w}-w^{1})\|_{\bar{0}}^{2})+\frac{1}{2}C_{4}^{2}T_{f}k^{4}\right]\exp(C_{1}\gamma_{0}^{-1}T_{f}).
     \end{equation*}

     Since $\underset{\widehat{w}\in \mathcal{W}_{h}}{\inf}\|w_{h}^{1}-\widehat{w}\|_{\bar{0}}^{2}=0$, taking the infimum over $\widehat{w}\in \mathcal{W}_{h}$ of this inequality and using estimate $(\ref{57})$, this results in
     \begin{equation*}
     \|w_{h}^{n+1}-w^{n+1}\|_{\bar{0}}^{2}+\frac{C_{1}}{\gamma_{0}\gamma_{1}\nu}\||w_{h}^{n+1}-w^{n+1}|\|_{A}^{2}\leq \frac{4C_{1}}{\gamma_{0}\gamma_{1}}
     \left[C_{3}^{2}k^{-2}h^{2p+2}\|w^{1}\|_{\overline{p+1}}^{2}+\frac{1}{2}C_{4}^{2}T_{f}k^{4}\right]\exp(C_{1}\gamma_{0}^{-1}T_{f}).
     \end{equation*}
     Utilizing the time step restriction $(\ref{sr})$ corresponding to the extreme case (i.e., $k=C_{sr}h$), this estimate becomes
     \begin{equation*}
     \|w_{h}^{n+1}-w^{n+1}\|_{\bar{0}}^{2}+\frac{C_{1}}{\gamma_{0}\gamma_{1}\nu}\||w_{h}^{n+1}-w^{n+1}|\|_{A}^{2}\leq \frac{4C_{1}}{\gamma_{0}\gamma_{1}}
     \left[(C_{3}C_{sr}^{-1})^{2}h^{2p}\|w^{1}\|_{\overline{p+1}}^{2}+\frac{1}{2}C_{4}^{2}T_{f}k^{4}\right]\exp(C_{1}\gamma_{0}^{-1}T_{f})\leq
     \end{equation*}
     \begin{equation*}
     \widehat{C}(k^{4}+h^{2p})\leq\widehat{C}(k^{2}+h^{p})^{2},
     \end{equation*}
      for $n=0,1,...,N-1$, where all the constants are absorbed into a positive constant $\widehat{C}$. The proof of Theorem $\ref{t2}$ is completed.
       \end{proof}

      \section{Numerical experiments}\label{sec4}

      This section simulates a modified Lax-Wendroff/interpolation approach with finite element method $(\ref{28})$-$(\ref{31})$ for solving a three-dimensional system of geological structure deformation problem $(\ref{1})$ subjects to initial-boundary conditions $(\ref{2})$-$(\ref{3})$. Two numerical examples are performed to confirm the theory and to show the utility and efficiency of the proposed approach $(\ref{28})$-$(\ref{31})$. Additionally, the new algorithm is used to assess and predict landslides occurred in the west and center regions in Cameroon from October $2019$ to November $2024$, namely in Dschang cliff, Mbankolo and Gouache cities. Moreover, the numerical results (displacement $(w_{h})$ and stress tensor $(\kappa_{h})$) obtained from the proposed computational technique using the data taken in these areas should provide useful information on some natural disasters in Cameroon (also worldwide) which would allow people to be informed about the risk zones due to landslides.

      \subsection*{Stability analysis and convergence order}

      To check the stability and convergence order of the proposed computational technique $(\ref{28})$-$(\ref{31})$, we set $h\in\{3^{-l},\text{\,\,}l=3,4,5,6\}$, where $h=\max\{h_{T},\text{\,\,}T\in\Pi_{h}\}$ and $\Pi_{h}$ denotes the triangulation of $\overline{\Omega}$. In addition, a uniform time step $k=3^{-m}$, for $m=5,6,7,8$, is used. We compute the errors: $w_{h}^{n}-w^{n}$ and $\kappa_{h}^{n}-\kappa^{n}$, at time $t_{n}$ using the norms $\||\cdot|\|_{\bar{0},\infty}$ and $\||\cdot|\|_{*,\infty}$, defined as
      \begin{equation*}
       \||w|\|_{\bar{0},\infty}=\underset{0\leq n\leq N}{\max}\|w^{n}\|_{\bar{0}},\text{\,\,\,}\forall w\in\mathcal{W}\text{\,\,\,and\,\,\,}\||\tau|\|_{*,\infty}=\underset{0\leq n\leq N}{\max}\|\tau^{n}\|_{*},\text{\,\,\,}\forall \tau\in\mathcal{M}.
         \end{equation*}
       Since the three-dimensional system of elastodynamic equations $(\ref{1})$ is complex, determine the analytical solution of the initial-boundary value problem $(\ref{1})$-$(\ref{3})$ is too difficult and sometimes impossible. Thus, the temporal errors are computed assuming that the analytical solution equals to the numerical solution with the time step $k=3^{-9}$, while the spatial errors are calculated assuming that the exact solution is the approximate one obtained with the time step $k=3^{-7}$. Finally, the space convergence order, $CO(h)$, of the developed numerical scheme is estimated using the formula
         \begin{equation*}
          CO(h)=\frac{\log\left(\frac{\||w_{3h}-w|\|_{\bar{0},\infty}}{\||w_{h}-w|\|_{\bar{0},\infty}}\right)}{\log(3)},\text{\,\,}
          \frac{\log\left(\frac{\||\tau_{3h}-\tau|\|_{*,\infty}}{\||\tau_{h}-\tau|\|_{*,\infty}}\right)}{\log(3)},
         \end{equation*}
          where, $z_{h}$ and $z_{3h}$ are the spatial errors associated with the grid sizes $h$ and $3h$, respectively, whereas the convergence rate in time, $CO(k)$, is computed utilizing the formula
         \begin{equation*}
          CO(k)=\frac{\log\left(\frac{\||w_{3k}-w|\|_{\bar{0},\infty}}{\||w_{k}-w|\|_{\bar{0},\infty}}\right)}{\log(3)},\text{\,\,}
          \frac{\log\left(\frac{\||\tau_{3k}-\tau|\|_{*,\infty}}{\||\tau_{k}-\tau|\|_{*,\infty}}\right)}{\log(3)},
         \end{equation*}
         where $z_{3k}$ and $z_{k}$ represent the errors in time corresponding to time steps $3k$ and $k$, respectively. It's worth recalling that the numerical computations are performed utilizing MATLAB R$2007b$ with built-in function $plot(\cdot)$ and $surf(\cdot)$ while the CPU is obtained using the command "cputime". Operating system and Hardware specifications of the computer: Window 11 Home Single Language, $64$-bit operating system, $x64$-based processor, RAM (4.00 GB) and 11th Gen Intel(R) Core(TM) $i3$-$1115G4@3.00$GHz $3.00$GHz.\\

          $\bullet$ \textbf{Example 1}. We consider the three-dimensional system of tectonic deformation equations $(\ref{1})$ defined on the domain $\overline{\Omega}=[0,\text{\,}1]^{3}$. The final time $T_{f}=1$. The values of physical parameters are: $\nu=1$, $\alpha=1$, $E=2.5$,  $\mu=\frac{ E}{2(1+\alpha)}=0.625$, $\lambda=\frac{\alpha E}{(1+\alpha)(1-2\alpha)}=-1.25$ and $g_{c}=\frac{1}{\pi}$. The function $g_{0}(x)=[\sin(\pi x_{1})\sin(\pi x_{2})\sin(\pi x_{3})]^{2}$. The initial and boundary conditions are given as: $w_{0}=w_{1}=\vec{0}$, on $\overline{\Omega}$ and $w(t)|_{\Gamma}=\vec{0}$, for every $t\in[0,\text{\,}1]$. We assume that the landslides started at the focus $B(x_{c},\text{\,}r_{0})$, where $x_{c}=(0.5,0.5,0.5)^{t}$ and $r_{0}=3^{-3}$. We set $C_{sr}=1$ be the positive constant defined in estimates $(\ref{sr})$.\\
         \text{\,}\\
         \textbf{Table 1.} $\label{T1}$ Convergence order $CO(h)$ of the constructed modified Lax-Wendroff/interpolation approach with finite element method $(\ref{28})$-$(\ref{31})$ using time step $k=3^{-5}$ and different space steps $h$, satisfying condition $(\ref{sr})$.
          \begin{equation*}
          \begin{array}{c c}
          \text{\,new algorithm,\,\,where\,\,}k=3^{-5}& \\
           \begin{tabular}{ccccccc}
            \hline
            $h$ &  $\||w_{h}-w|\|_{\bar{0},\infty}$ & $CO(h)$ & CPU(s) & $\||\kappa_{h}-\kappa|\|_{*,\infty}$ & $CO(h)$ & CPU(s)\\
             \hline
            $3^{-3}$ &  $1.0521\times10^{-3}$ & .... & 1.3631 & $1.4537\times10^{-4}$ & .... & 4.3967\\

            $3^{-4}$ &  $2.5119\times10^{-4}$ & 3.0021 &  3.5913 & $3.4762\times10^{-5}$ & 2.9987 & 13.1901\\

            $3^{-5}$ &  $6.0536\times10^{-5}$ & 2.9824 &  8.5346 & $8.2992\times10^{-6}$ & 3.0021 & 28.1213\\

            $3^{-6}$ &  $1.3929\times10^{-5}$ & 3.0794 & 19.8353 & $1.9388\times10^{-6}$ & 3.0476 & 84.3639 \\
            \hline
          \end{tabular} &
          \end{array}
          \end{equation*}
          \text{\,}\\
          $\bullet$ \textbf{Example 2}. Consider the three-dimensional system of geological structure deformation problem $(\ref{1})$ defined on the region $\overline{\Omega}\times[0,\text{\,}T_{f}]=[-1,\text{\,}1]^{3}\times[0,\text{\,}2]$. The physical parameters are given as: $\nu=1$, $\alpha=0.25$, $E=2.5$,  $\mu=\frac{ E}{2(1+\alpha)}=1$, $\lambda=\frac{\alpha E}{(1+\alpha)(1-2\alpha)}=1$ and $g_{c}=\frac{1}{\pi}$. The function $g_{0}(x)=[\sin(\pi x_{1})\sin(\pi x_{2})\sin(\pi x_{3})]^{2}$. Both initial and boundary conditions are defined as: $w_{0}=w_{1}=\vec{0}$, on $\overline{\Omega}$ and $w(t)|_{\Gamma}=\vec{0}$, for every $t\in[0,\text{\,}2]$. It is assumed that the landslides started at the focus $B(x_{c},\text{\,}r_{0})$, where $x_{c}=(0,0,0)^{t}$ and $r_{0}=3^{-3}$. We take $C_{sr}=3^{-1}$ be the positive constant defined in estimates $(\ref{sr})$.\\

          \textbf{Table 2.} $\label{T2}$ Convergence order $CO(k)$ of the proposed computational technique $(\ref{28})$-$(\ref{31})$ with space step $h=3^{-3}$ and varying time steps $k$, satisfying requirement $(\ref{sr})$.
           \begin{equation*}
          \begin{array}{c c}
          \text{\,new algorithm,\,\,where\,\,}h=3^{-3}& \\
           \begin{tabular}{ccccccc}
            \hline
            $k$ &  $\||w_{h}-w|\|_{\bar{0},\infty}$ & $CO(k)$ & CPU(s) & $\||\kappa_{h}-\kappa|\|_{*,\infty}$ & $CO(k)$ & CPU(s)\\
             \hline
            $3^{-5}$ &  $1.8279\times10^{-3}$ & .... & 1.4607  & $8.0203\times10^{-5}$ & .... & 3.8438\\

            $3^{-6}$ &  $7.0450\times10^{-4}$ & 1.9983 & 3.6820  & $3.0885\times10^{-5}$ & 2.0001 & 12.2117\\

            $3^{-7}$ &  $2.6635\times10^{-4}$ & 2.0386 & 8.2791  & $1.2018\times10^{-5}$ & 1.9783 & 27.4071\\

            $3^{-8}$ &  $9.5766\times10^{-5}$ & 2.1439 & 19.6492 & $4.3897\times10^{-6}$ & 2.1109 & 78.7666\\
            \hline
          \end{tabular} &
          \end{array}
          \end{equation*}
          \text{\,}\\
          Tables 1 $\&$ 2 show that the developed approach is spatial three-order accurate and second-order convergent in time.

        \subsection*{Landslides analysis in Dschang cliff, Mbankolo and Gouache (in Cameroon) from October $2019$ to November $2024$}

         The landslides model for this three-dimensional system of geological structure deformation is described as follows: the study areas are located in the west and center regions in Cameroon, specifically in Dschang (Dschang cliff), Yaounde (Mbankolo) and Bafoussam (Gouache district). The Dschang cliff has an elevation of $1400m$ (latitude $5^{o}24'18"N$, longitude $10^{o}00'57"E$) together with an annual rainfall intensity equals $2200mm$/year and a slope action of $7.6\%$. Some characteristics of the landscape consider the north-east valleys separated with plains (Figure $\ref{fig1}$, Figures 1(c) $\&$ 1(e)) while the Mbankolo geographical coordinates are between $3^{o}54'20"N$ to $11^{o}20'20"E$ for the hazard point upstream at altitude $823.8m$ and $3^{o}54'36"N$ to $11^{o}29'27E$ for the hazard point downstream at altitude $780m$. The landscape is characterized by north-east valleys separated with structural plains (Figure $\ref{fig1}$, Figure 1(f)). The Gouache landscape is characterized by the north-east valleys separated with mountains whereas the surrounding reliefs of the hazard site are at altitude nearly $1532$m (latitude $5^{o}24'18"N$, longitude $10^{o}00'57"E$). In $2019$, a rainfall's rise has been recorded from January to September in the landslide site with highest values of $592mm$ in September and $544mm$ in October while the lowest rainfall assessed in December of this year was indicating $0.5mm$ (Figure $\ref{fig1}$, Figure 1(g)). In the simulations, the initial and boundary conditions for displacement are equal to zero; the constant $C_{sr}$ in the time step limitation $(\ref{sr})$ equals $3^{-1}$, the mesh size in $x$-, $y$- and $z$-directions $h=3^{-3}$ and time step $k=3^{-5}$ are chosen so that requirement $(\ref{sr})$ is satisfied. The periods of events are represented by the time interval $[0,\text{\,}T_{f}]$ (time in days) where: (a) $T_{f}=3$ corresponds to $5$-$7$ November $2024$ (Dschang cliff landslides), (b) $T_{f}=3$ indicates $6$-$8$ October $2023$ (Mbankolo landslides) and (c) $T_{f}=1$ denotes $29$ October $2024$ (Gouache landslides). We should compute the norm of the displacement, $\||w_{h}|\|_{\bar{0},\infty}$, and the one of the symmetric stress tensor elements, $\||\kappa_{h,ij}|\|_{0,\infty}$, for $i,j=1,2,3$. It's important to remind that $\kappa_{h,ij}$=$\kappa_{h,ji}$. Because of the graphs similarity and for the sake of reducing the paper length, for each discussed case we displayed in the figures only one component of displacement and two elements of the symmetric stress tensor. Indeed, display three components of displacement, six elements of stress tensor and the graphs related to stability and convergence of the proposed numerical approach $(\ref{28})$-$(\ref{31})$ should result in more than $46$ graphs which seem to be too much.\\

        $\bullet$ \textbf{Example 3}. Consider the three-dimensional system of elastodynamic equations $(\ref{1})$ defined on $\overline{\Omega}=[-1,\text{\,}1]^{3}$, with final time $T_{f}$. The physical parameters are chosen as: $\nu=1$, $\alpha=0.25$, $E=2.5$, $\mu=\frac{ E}{2(1+\alpha)}=1$, $\lambda=\frac{\alpha E}{(1+\alpha)(1-2\alpha)}=1$ and $C_{sr}=3^{-1}$. The function $g_{0}(x)=[\sin(\pi x_{1})\sin(\pi x_{2})\sin(\pi x_{3})]^{2}$. The initial and boundary conditions are defined as: $w_{0}=w_{1}=\vec{0}$, $\kappa_{11}^{0}=2\times10^{-2},$ $\kappa_{22}^{0}=3\times10^{-1}$, $\kappa_{33}^{0}=1.5\times10^{-2}$, $\kappa_{12}^{0}=0.5\times10^{-3}$, $\kappa_{13}^{0}=0.75\times10^{-3}$, $\kappa_{23}^{0}=1.25\times10^{-3}$, on $\overline{\Omega}$ and $w(t)|_{\Gamma}=\vec{0}$, for every $t\in[0,\text{\,}T_{f}]$. It's assumed that the landslides started at the focus $B(x_{c},\text{\,}r_{0})$, where $x_{c}=(0,0,0)^{t}$, $r_{0}=3^{-2}$ and "$\times$" denotes the usual multiplication in $\mathbb{R}$.\\

         \textbf{Table 3.} $\label{T3}$ Displacement $w_{h}$ and symmetric stress tensor $\kappa_{h}=(\kappa_{h,ij})$, for $1\leq i,j\leq3$, corresponding to Dschang cliff landslides (from November $5$-$7$, $2024$) provided by the new algorithm utilizing a slope action $g_{c}=7.6\%$ and space step $h=3^{-3}$, where $T_{f}=3$ days.
          \begin{equation*}
          \begin{array}{c }
          \text{\,Developed approach,\,\,where\,\,}k=3^{-5} \\
           \begin{tabular}{cccccccc}
            \hline
            $h$ &  $\||w_{h}|\|_{\bar{0},\infty}$ & $\||\kappa_{h,11}|\|_{0,\infty}$ & $\||\kappa_{h,22}|\|_{0,\infty}$ & $\||\kappa_{h,33}|\|_{0,\infty,}$ & $\||\kappa_{h,12}|\|_{0,\infty}$ & $\||\kappa_{h,13}|\|_{0,\infty}$ & $\||\kappa_{h,23}|\|_{0,\infty}$\\
             \hline
            $3^{-3}$ & $0.1182$ & $0.6632$ & $0.8961$ & $4.2313$ & $0.0269$ & $0.2708$ & $0.3203$\\
            \hline
          \end{tabular}
          \end{array}
          \end{equation*}
           In addition, the symmetric stress tensor is given by
          \begin{equation*}
           \kappa_{h}=\begin{bmatrix}
              0.6632 & 0.0269 & 0.2708 \\
              0.0269 & 0.8961 & 0.3203 \\
              0.2708 & 0.3203 & 4.2313 \\
            \end{bmatrix}
          \end{equation*}
          \text{\,}\\
            \textbf{Table 4.} $\label{T4}$ Displacement $w_{h}$ and symmetric stress tensor $\kappa_{h}=(\kappa_{h,ij})$, for $1\leq i,j\leq3$, associated with Dschang cliff landslides (from November $5$-$7$, $2024$) provided by the new algorithm using altitude $g_{c}=1450-750=700m=0.7km$, space step $h=3^{-3}$ and time step $k=3^{-5}$, where $T_{f}=3$ days.
          \begin{equation*}
          \begin{array}{c }
          \text{\,Developed approach,\,\,where\,\,}k=3^{-5} \\
           \begin{tabular}{cccccccc}
            \hline
            $h$ &  $\||w_{h}|\|_{\bar{0},\infty}$ & $\||\kappa_{h,11}|\|_{0,\infty}$ & $\||\kappa_{h,22}|\|_{0,\infty}$ & $\||\kappa_{h,33}|\|_{0,\infty,}$ & $\||\kappa_{h,12}|\|_{0,\infty}$ & $\||\kappa_{h,13}|\|_{0,\infty}$ & $\||\kappa_{h,23}|\|_{0,\infty}$\\
             \hline
            $3^{-3}$ & $0.0783$ & $0.4412$ & $0.8961$ & $2.7684$ & $0.0174$ & $0.1774$ & $0.2092$\\
            \hline
          \end{tabular}
          \end{array}
          \end{equation*}
           Furthermore, the symmetric stress tensor is given by
          \begin{equation*}
           \kappa_{h}=\begin{bmatrix}
              0.4412 & 0.0174 & 0.1774 \\
              0.0174 & 0.8961 & 0.2092 \\
              0.1774 & 0.2092 & 2.7684 \\
            \end{bmatrix}
          \end{equation*}
          \text{\,}\\
            \textbf{Table 5.} $\label{T5}$ Displacement $w_{h}$ and symmetric stress tensor $\kappa_{h}=(\kappa_{h,ij})$, for $1\leq i,j\leq3$, associated with Mbankolo landslides (October $6$-$8$, $2023$) obtained from the new algorithm utilizing altitude $823.8-780=43.8m$ and grid space $h=3^{-3}$, where $T_{f}=3$ days.
          \begin{equation*}
          \begin{array}{c }
          \text{\,Developed approach,\,\,where\,\,}k=3^{-5} \\
           \begin{tabular}{cccccccc}
            \hline
            $h$ &  $\||w_{h}|\|_{\bar{0},\infty}$ & $\||\kappa_{h,11}|\|_{0,\infty}$ & $\||\kappa_{h,22}|\|_{0,\infty}$ & $\||\kappa_{h,33}|\|_{0,\infty,}$ & $\||\kappa_{h,12}|\|_{0,\infty}$ & $\||\kappa_{h,13}|\|_{0,\infty}$ & $\||\kappa_{h,23}|\|_{0,\infty}$\\
             \hline
            $3^{-3}$ & $0.1189$ & $0.6670$ & $0.8961$ & $4.2471$ & $0.0271$ & $0.2725$ & $0.3222$\\
            \hline
          \end{tabular}
          \end{array}
          \end{equation*}
           Additionally, the symmetric stress tensor is given by
          \begin{equation*}
           \kappa_{h}= \begin{bmatrix}
              0.6670 & 0.0271 & 0.2725 \\
              0.0271 & 0.8961 & 0.3222 \\
              0.2725 & 0.3222 & 4.2471 \\
            \end{bmatrix}
          \end{equation*}
          \text{\,}\\
          \textbf{Table 6.} $\label{T6}$ Displacement $w_{h}$ and stress tensor $\kappa_{h}=(\kappa_{h,ij})$, for $1\leq i,j\leq3$, corresponding to Gouache landslides (October $29$, $2019$) provided by the new algorithm using altitude $1532-1323=209m$ and mesh grid $h=3^{-3}$, where $T_{f}=1$ day.
          \begin{equation*}
          \begin{array}{c }
          \text{\,Developed approach,\,\,where\,\,}k=3^{-5} \\
           \begin{tabular}{cccccccc}
            \hline
            $h$ &  $\||w_{h}|\|_{\bar{0},\infty}$ & $\||\kappa_{h,11}|\|_{0,\infty}$ & $\||\kappa_{h,22}|\|_{0,\infty}$ & $\||\kappa_{h,33}|\|_{0,\infty,}$ & $\||\kappa_{h,12}|\|_{0,\infty}$ & $\||\kappa_{h,13}|\|_{0,\infty}$ & $\||\kappa_{h,23}|\|_{0,\infty}$\\
             \hline
            $3^{-3}$ & $0.0132$ & $0.0909$ & $0.8961$ & $0.4738$ & $0.0033$ & $0.0302$ & $0.0356$\\
            \hline
          \end{tabular}
          \end{array}
          \end{equation*}
           In addition, the symmetric stress tensor is given by
           \begin{equation*}
           \kappa_{h}= \begin{bmatrix}
              0.0909 & 0.0033 & 0.0302\\
              0.0033 & 0.8961 & 0.0356\\
              0.0302 & 0.0356 & 0.4738\\
            \end{bmatrix}
          \end{equation*}
          \text{\,}\\
          The approximate solutions obtained from the constructed technique $(\ref{28})$-$(\ref{31})$ are displayed in Figures $\ref{fig2}$. A time step $k=3^{-5}$, and space step $h=3^{-3},$ which satisfy the time step restriction $(\ref{sr})$ are used. Figures $\ref{fig3}$-$\ref{fig4}$ show that the displacement and elements of stress tensor propagate with almost a perfectly value at different positions while the error associated with the displacement approaches zero. Thus, under the time step restriction $(\ref{sr})$, the numerical solutions could not increase with time. In addition, It follows from \textbf{Tables 1-2} that the errors associated with the displacement and stress tensor are second-order in time and spatial third-order. This suggests that the new approach $(\ref{28})$-$(\ref{31})$ is temporal second-order convergent and space third-order accurate. Furthermore, \textbf{Tables 1-2} and Figures $\ref{fig2}$ indicate that the numerical solutions do not increase with time and converge to the analytical one. Specifically, they show that the constructed algorithm $(\ref{28})$-$(\ref{31})$ is not unconditionally unstable, but stability depends on the parameters $h$ and $k$. On the other hand, Tables \textbf{Tables 3-6} present some values of the displacement and symmetric stress tensor elements associated with the Dschang cliff, Mbankolo and Gouache landslides. The information provided by these Tables and the graphs displayed in Figures $\ref{fig3}$-$\ref{fig4}$ indicate that the proposed numerical approach $(\ref{28})$-$(\ref{31})$ should be considered as a powerful tool to predict runout and delineate the hazardous landslide areas.\\

          \begin{figure}
         \begin{center}
         Stability of the new algorithm for three-dimensional system of tectonic deformation model.
         \begin{tabular}{c c}
         \psfig{file=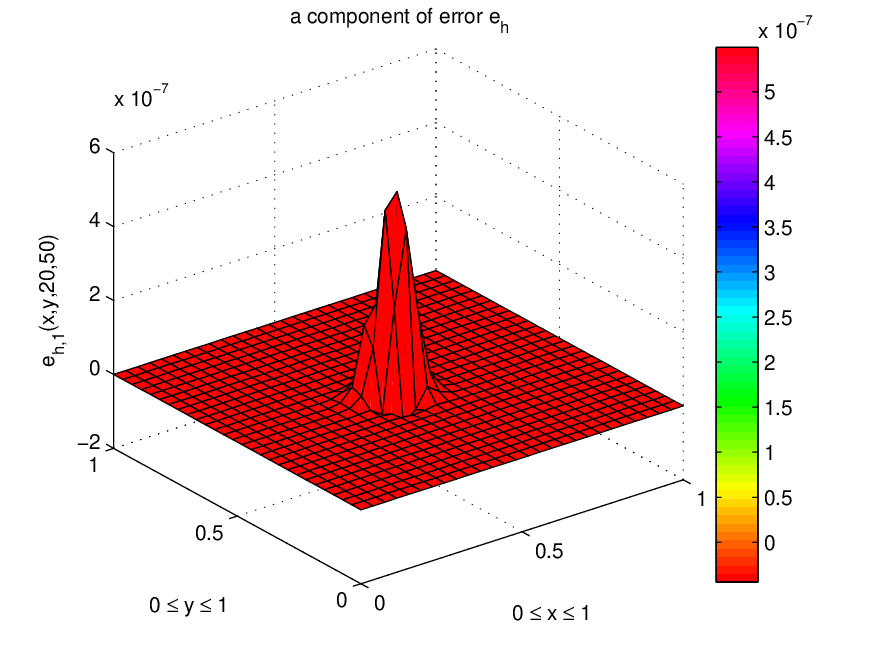,width=7cm} & \psfig{file=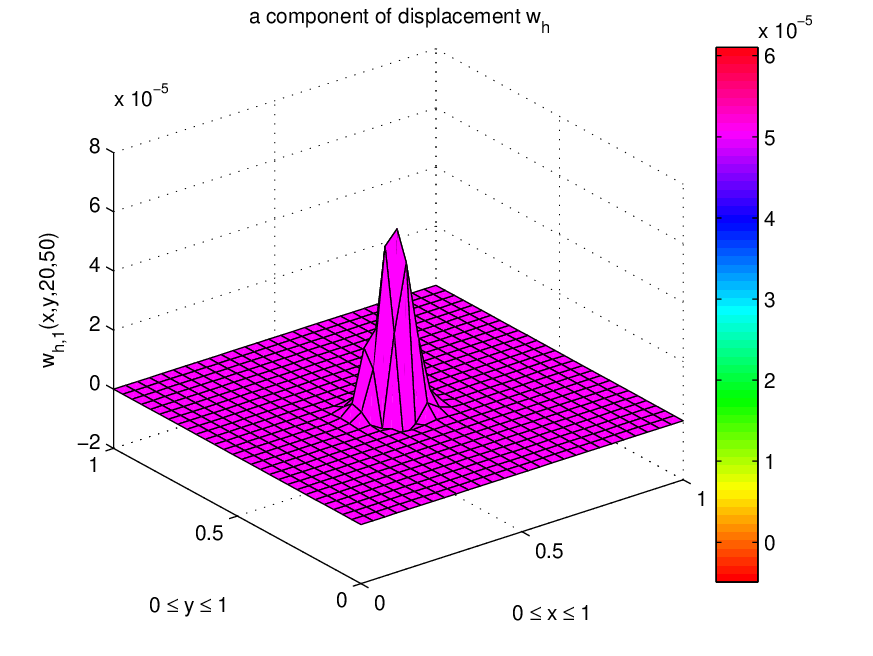,width=7cm}\\
         \psfig{file=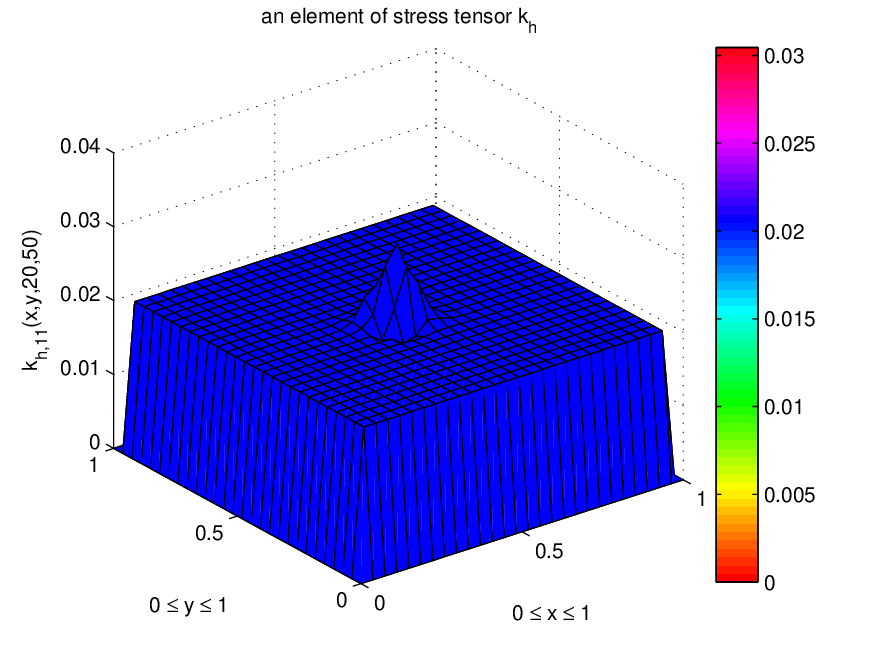,width=7cm} & \psfig{file=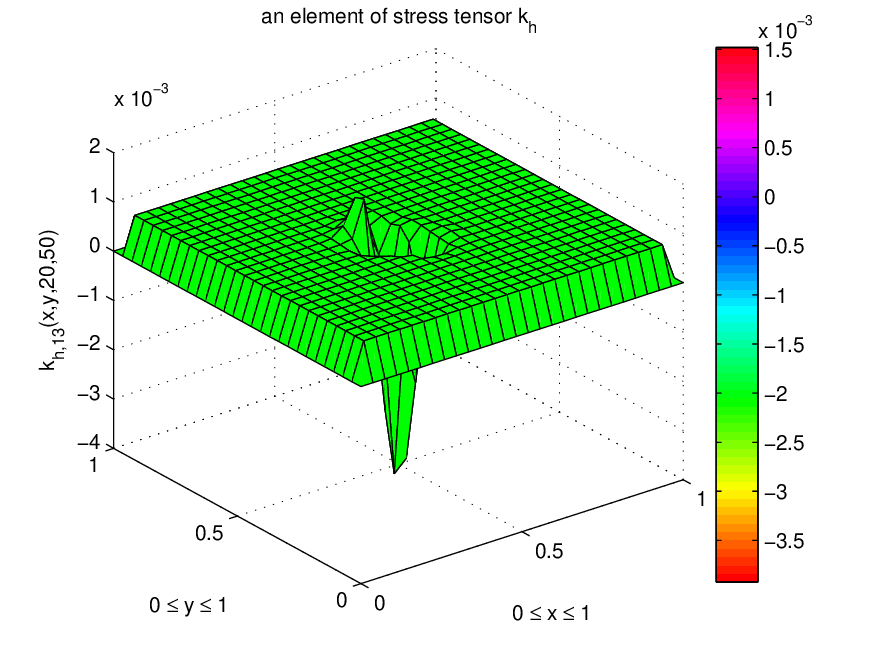,width=7cm}\\
         \psfig{file=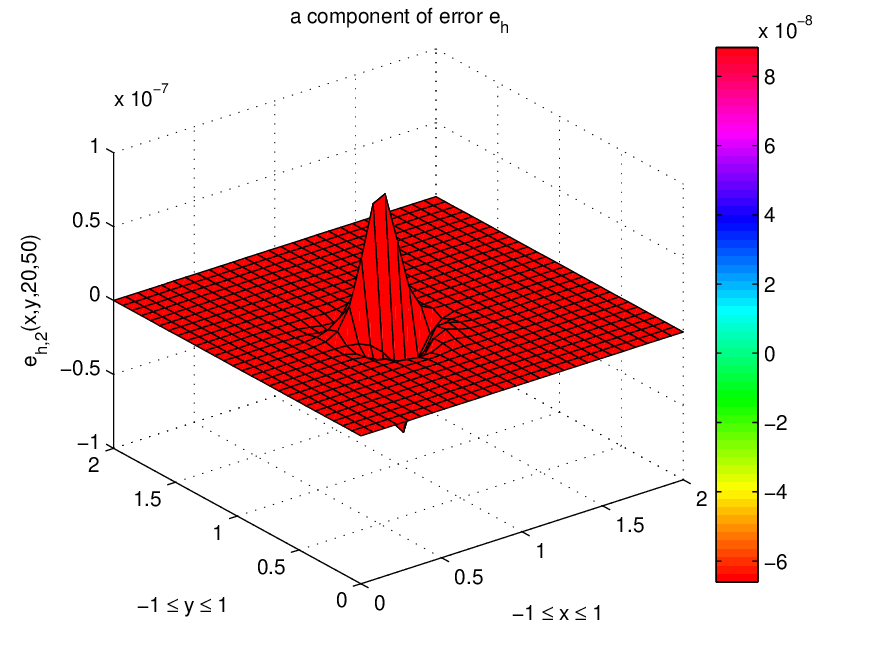,width=7cm} & \psfig{file=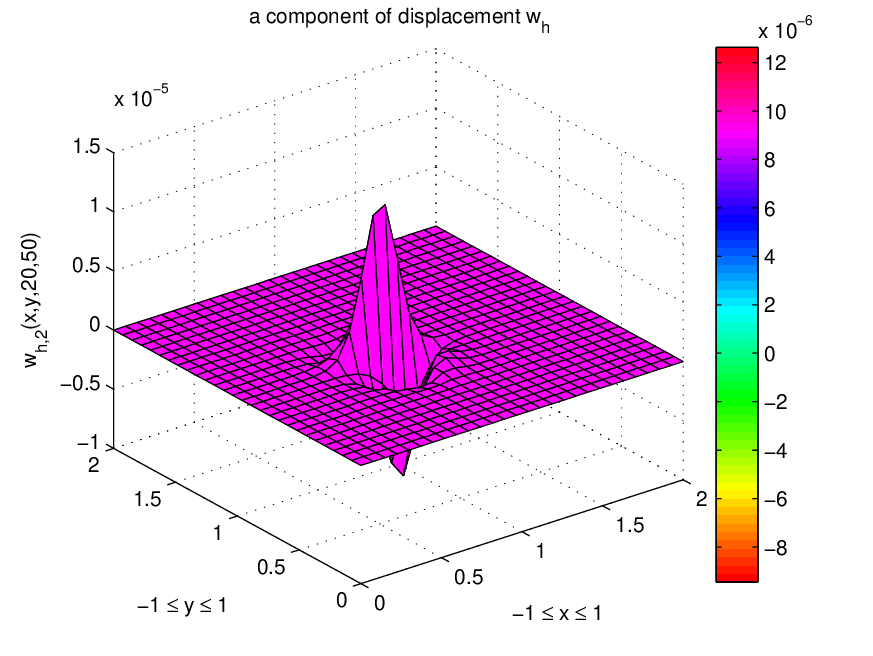,width=7cm}\\
         \psfig{file=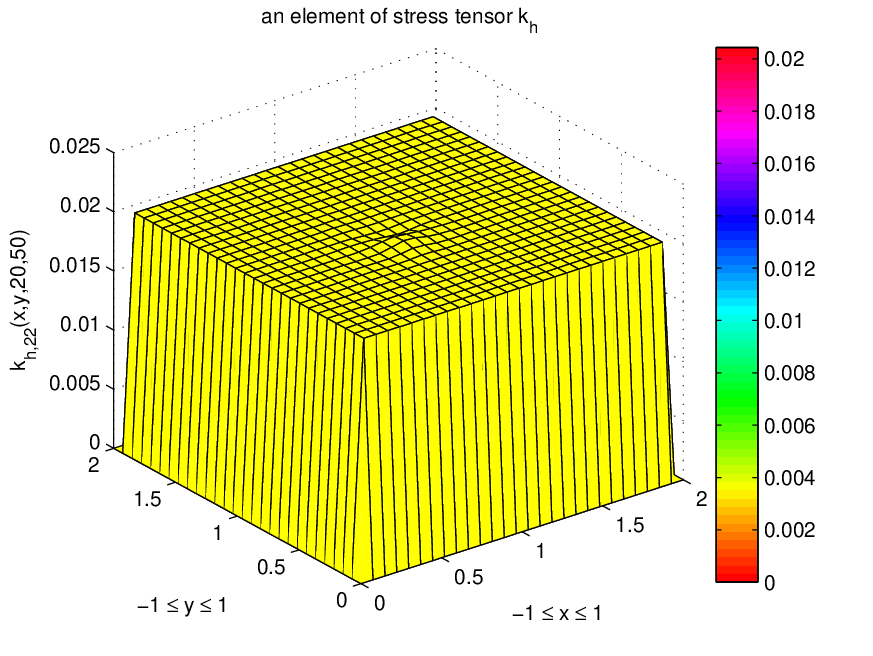,width=7cm} & \psfig{file=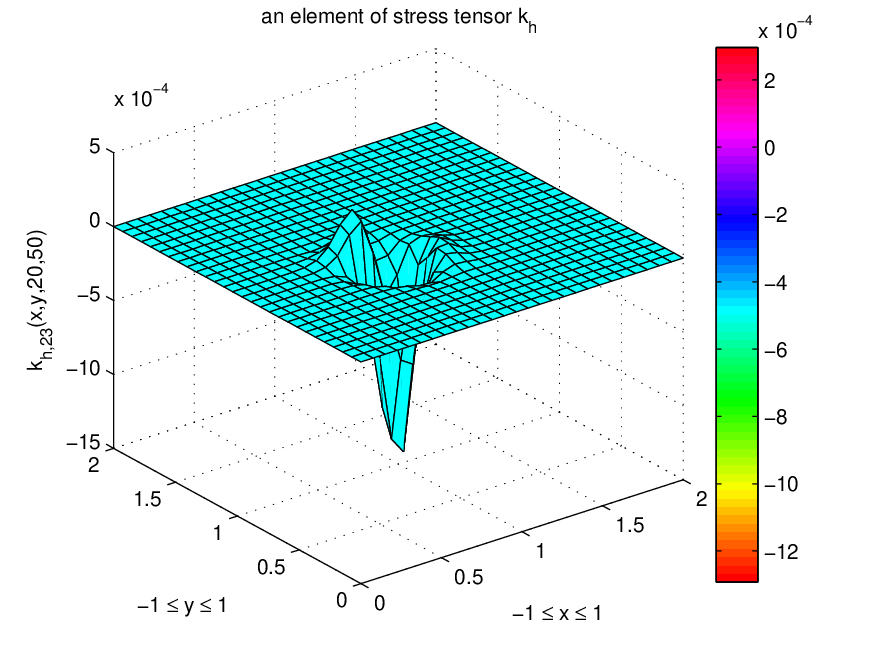,width=7cm}
         \end{tabular}
        \end{center}
        \caption{Graphs of error $(e_{h})$, displacement $(w_{h})$ and stress tensor $(\kappa_{h})$ corresponding to Example 1 (first four figures) and Example 2 (last four figures).}
        \label{fig2}
        \end{figure}

    \begin{figure}
     \begin{center}
       Analysis of Dschang cliff landslides with a slope action: $g_{c}=7.6\%$ and an altitude difference: $g_{c}=0.7km$.
      \begin{tabular}{c c}
         \psfig{file=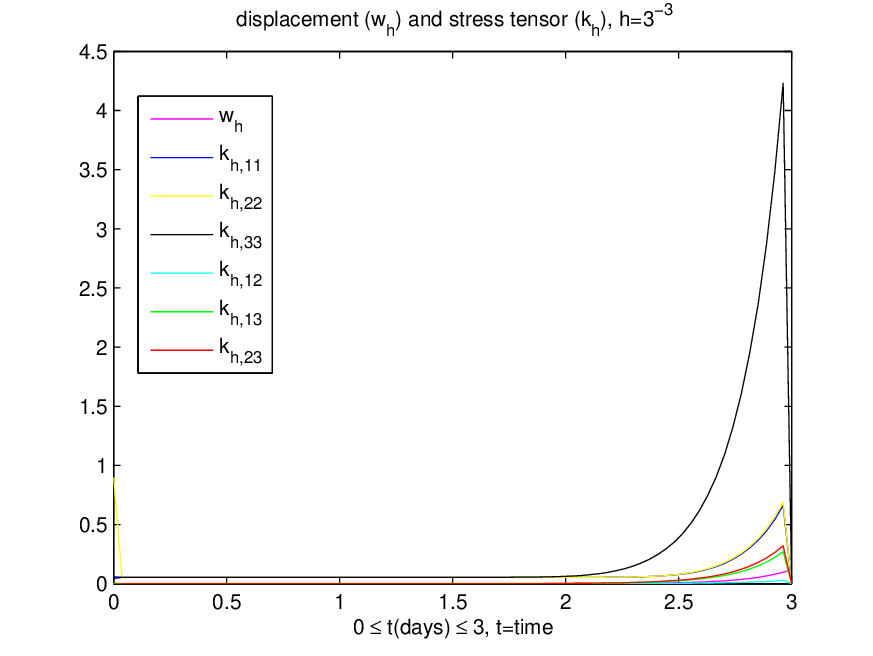,width=7cm} & \psfig{file=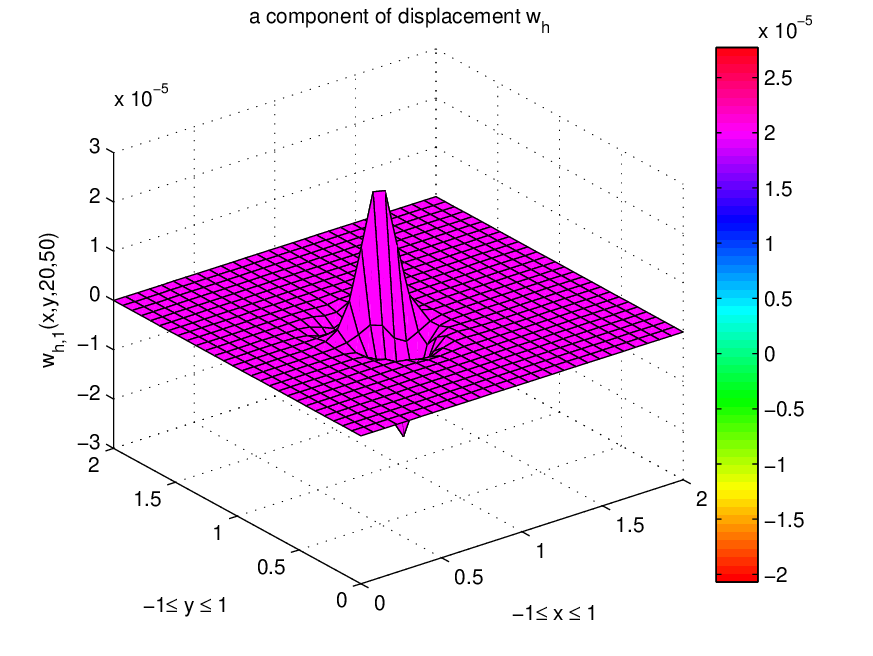,width=7cm}\\
         \psfig{file=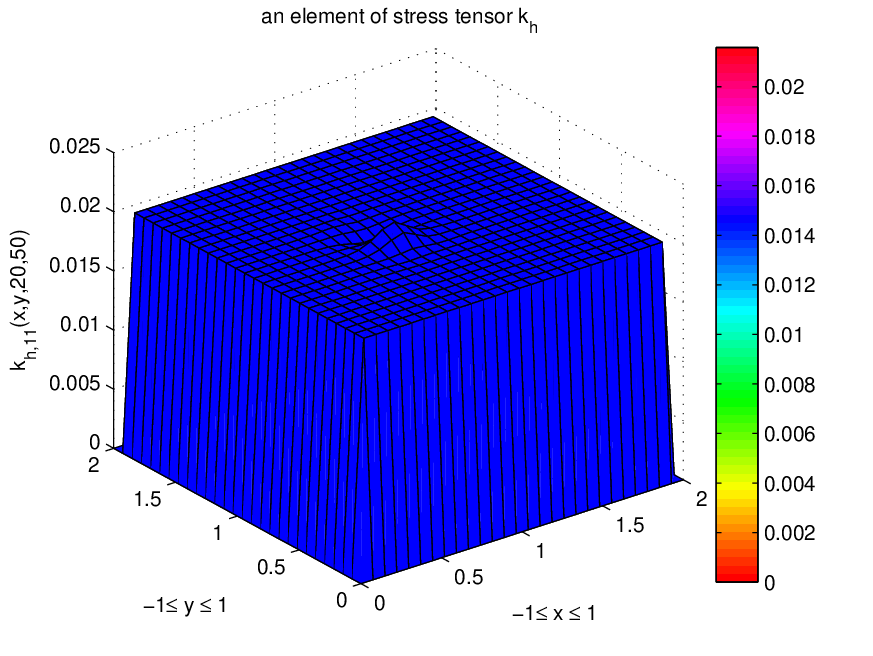,width=7cm} & \psfig{file=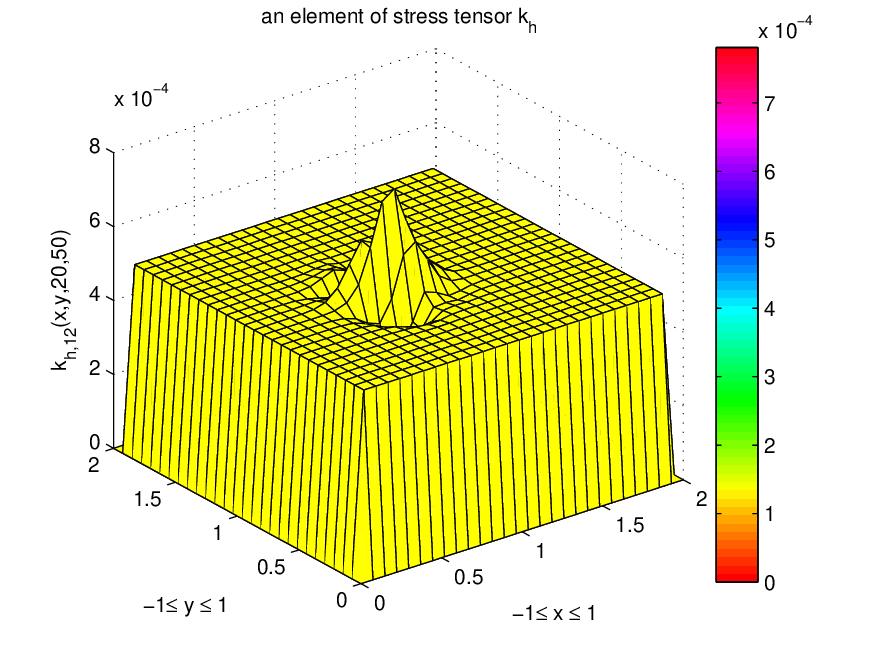,width=7cm}\\
         \psfig{file=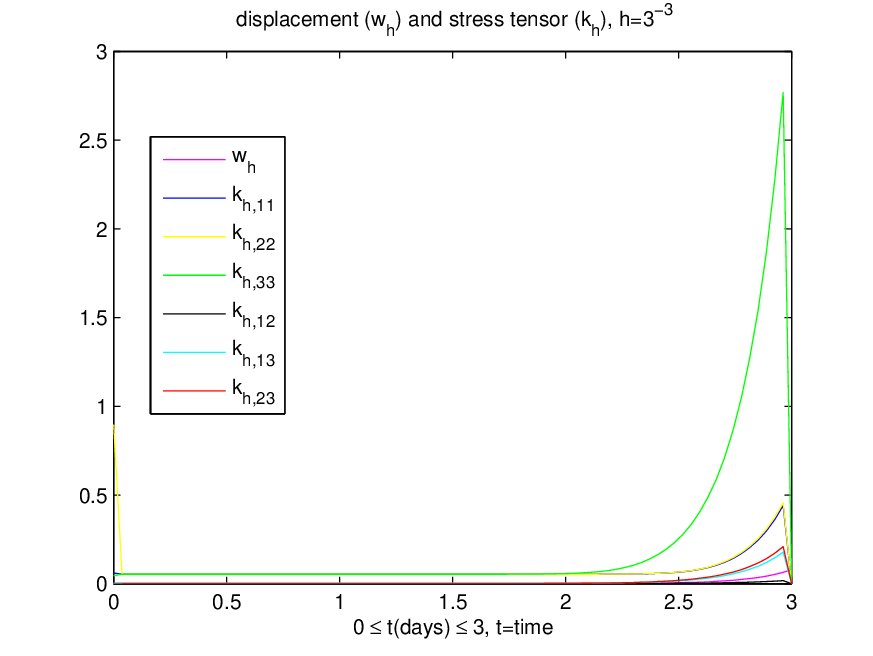,width=7cm} & \psfig{file=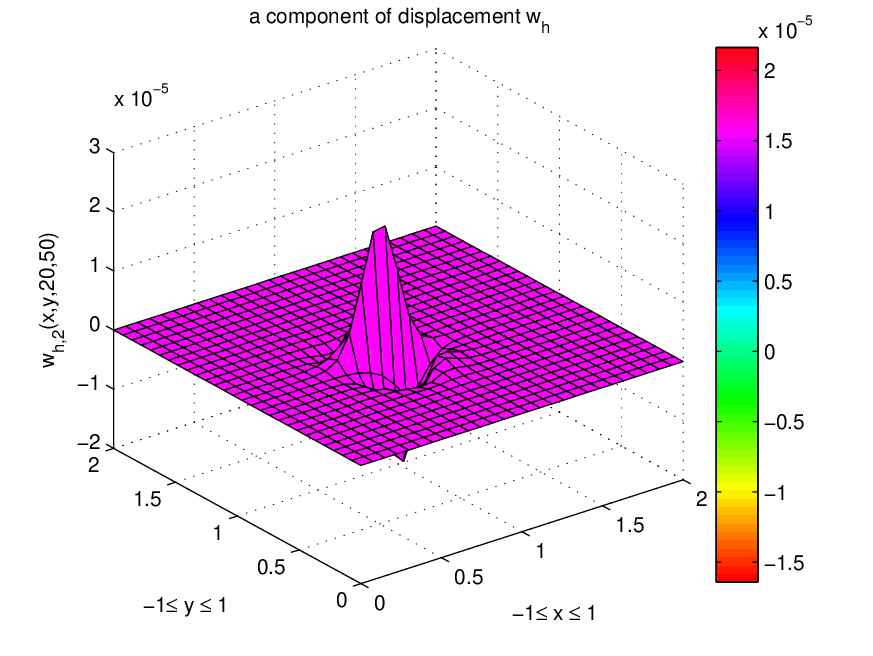,width=7cm}\\
         \psfig{file=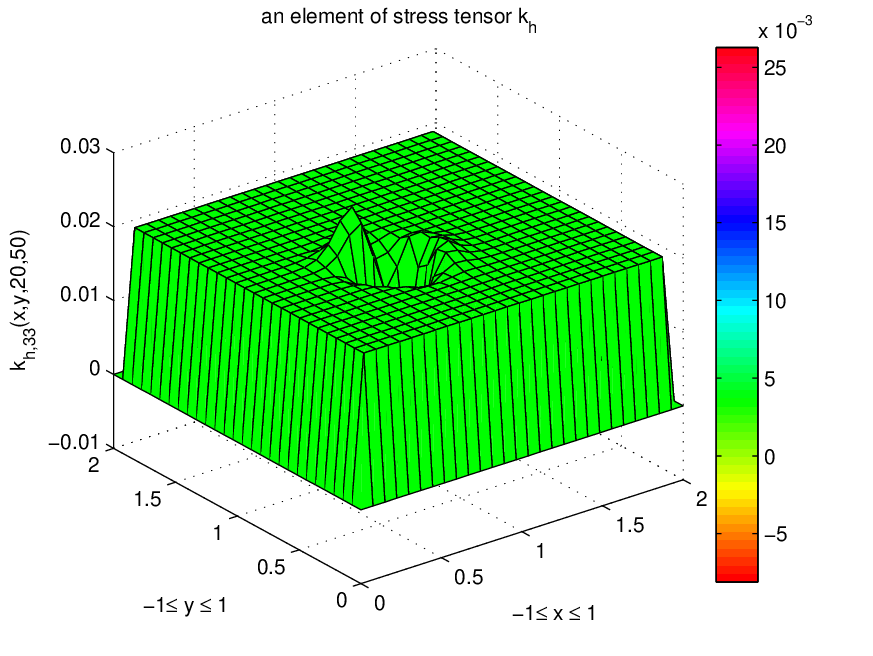,width=7cm} & \psfig{file=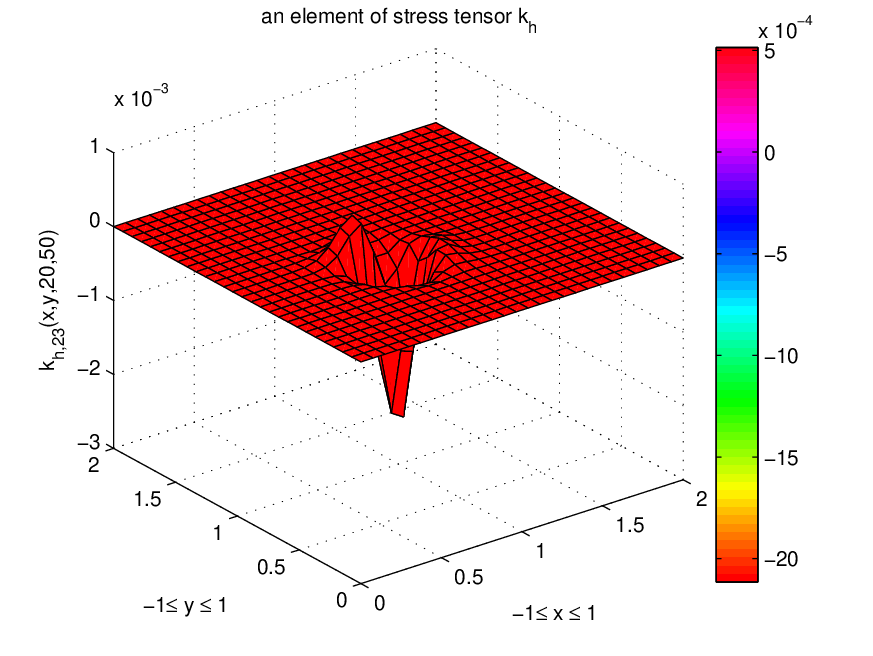,width=7cm}
         \end{tabular}
        \end{center}
         \caption{Graphs of displacement and stress tensor related to Dschang cliff landslides: first four figures for $g_{c}=7.6\%$ and last four figures for $g_{c}=0.7km$.}
          \label{fig3}
          \end{figure}

      \begin{figure}
     \begin{center}
       Analysis of Mbankolo and Gouache landslides with an altitude difference $g_{c}=0.209km$ and $g_{c}=0.0438km$.
      \begin{tabular}{c c}
         \psfig{file=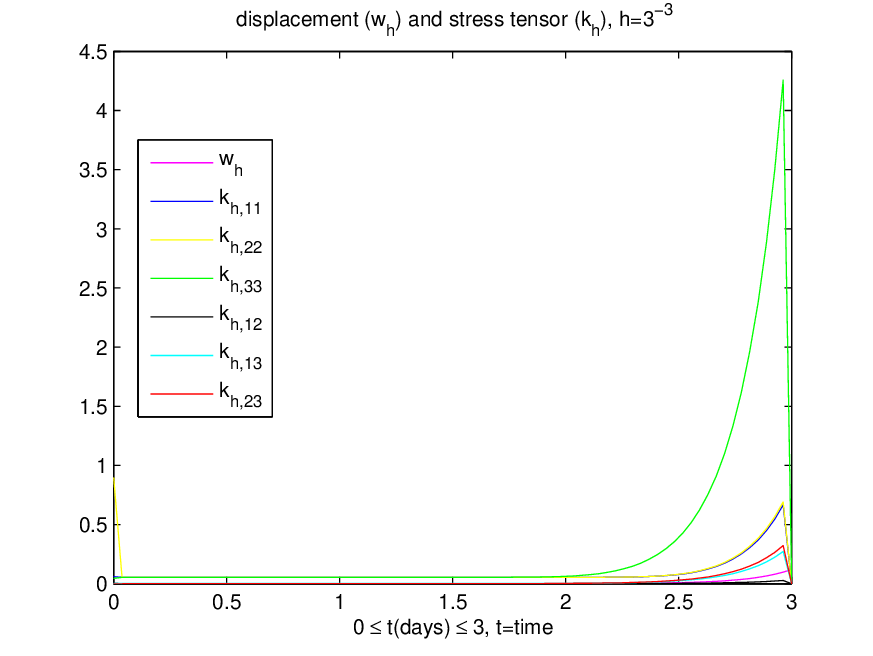,width=5.5cm} & \psfig{file=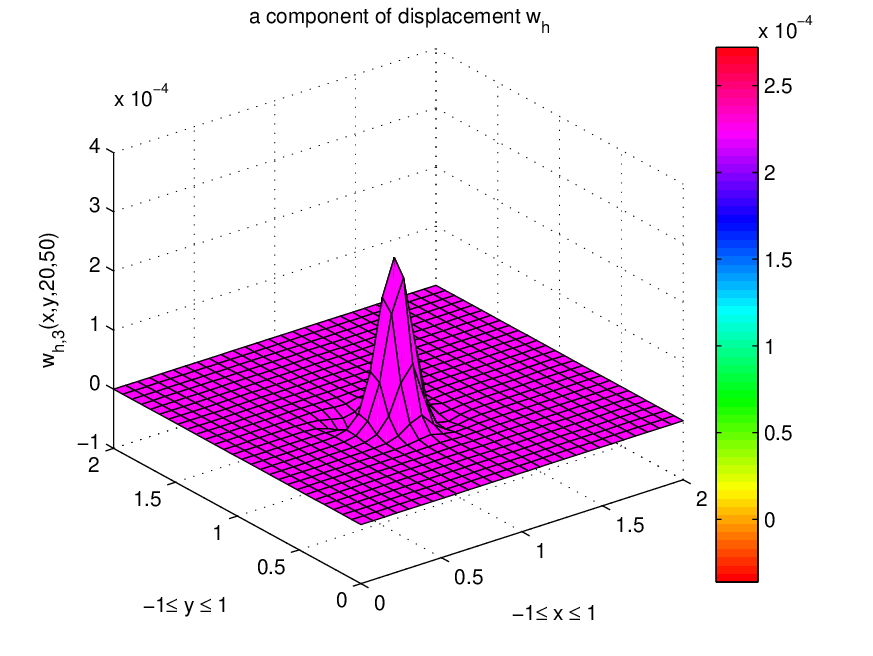,width=5.5cm}\\
         \psfig{file=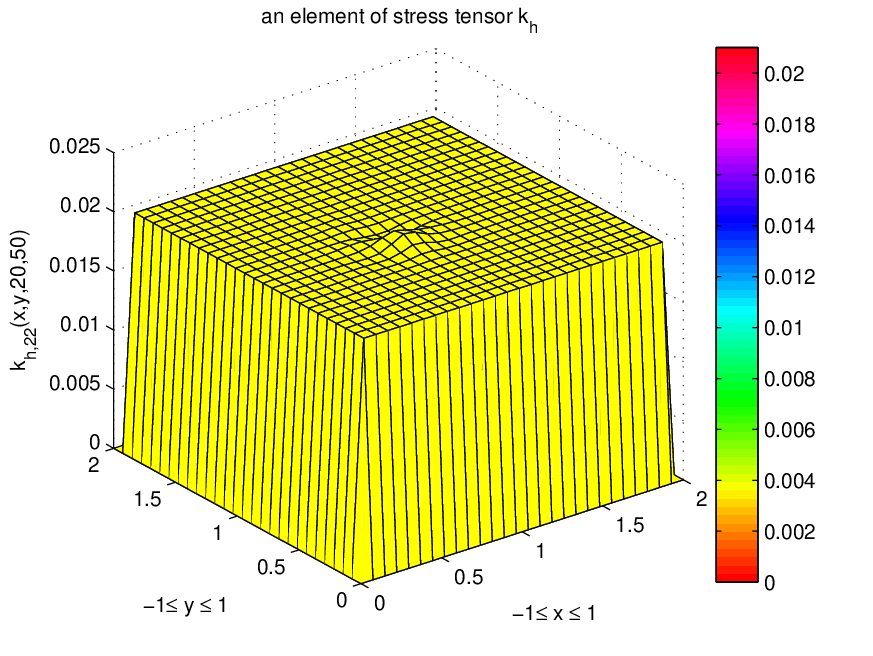,width=5.5cm} & \psfig{file=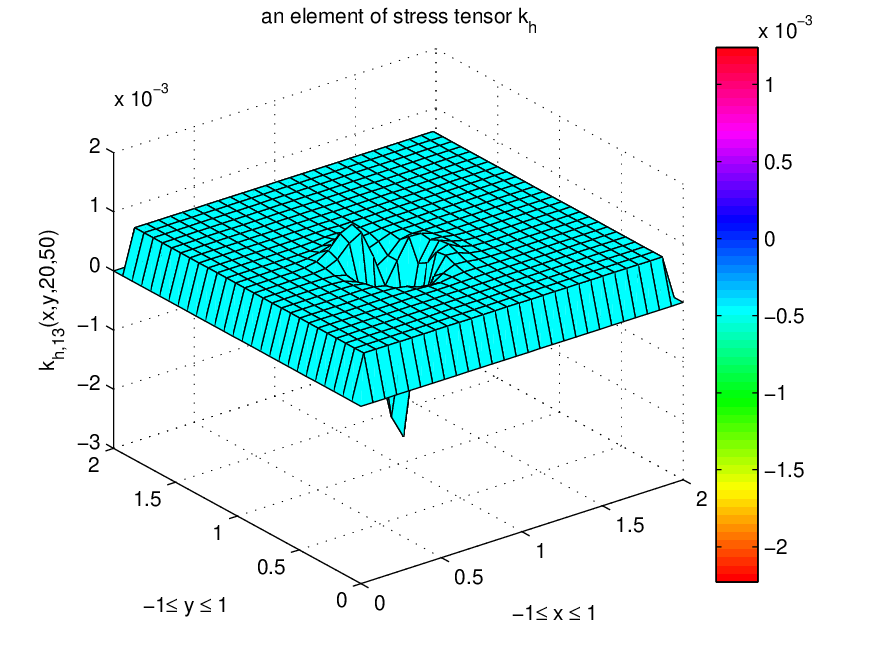,width=5.5cm}\\
         \psfig{file=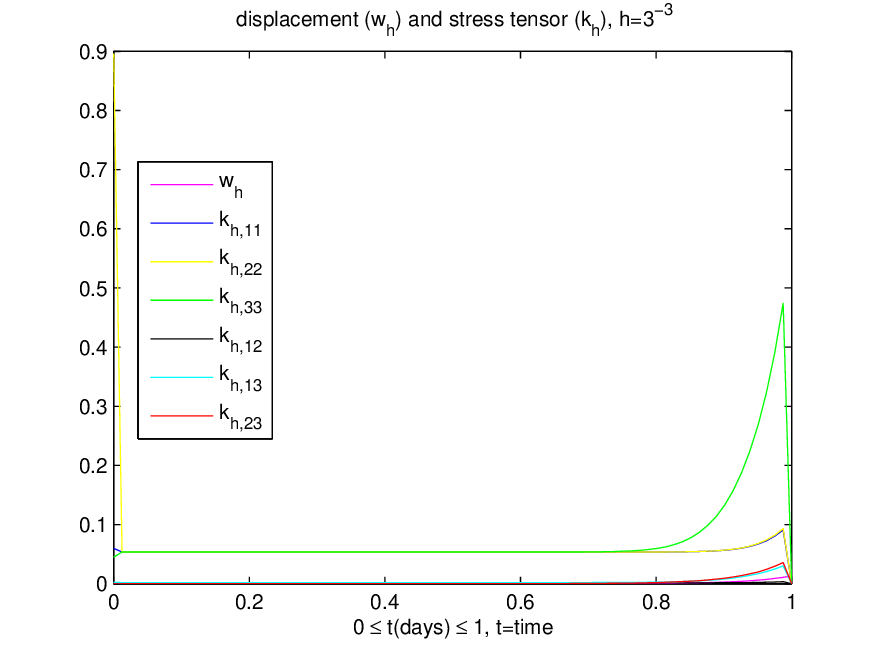,width=5.5cm} & \psfig{file=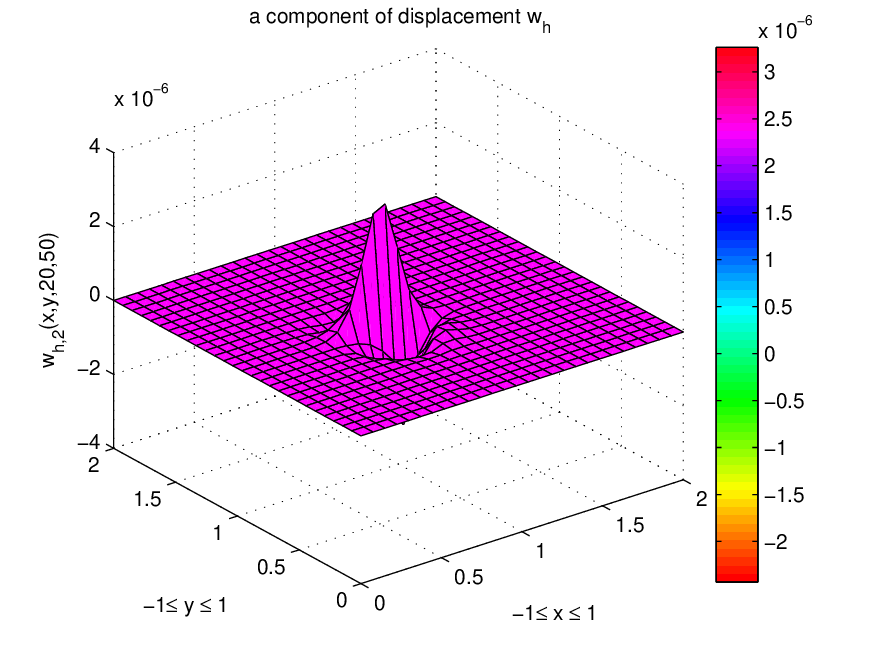,width=5.5cm}\\
         \psfig{file=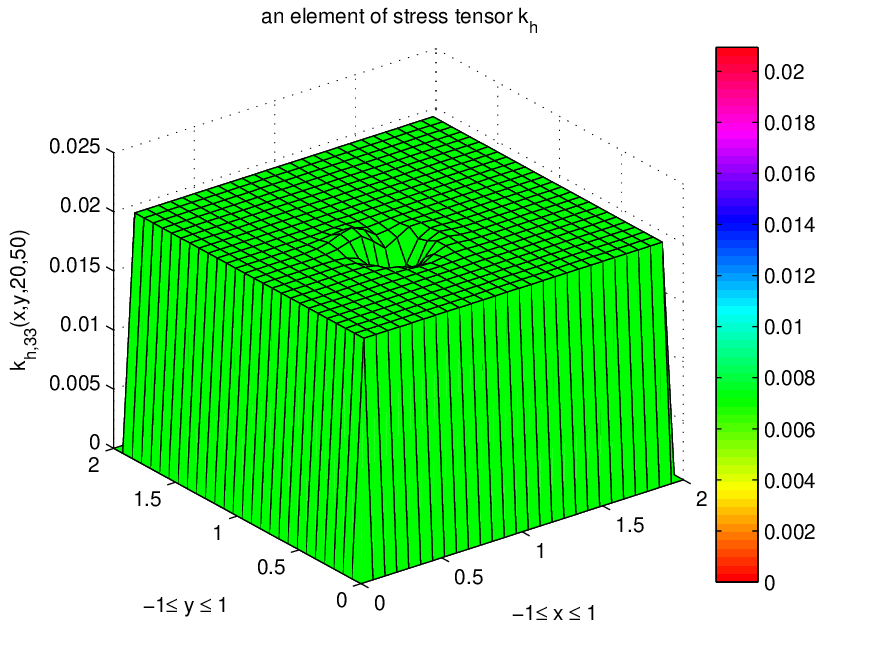,width=5.5cm} & \psfig{file=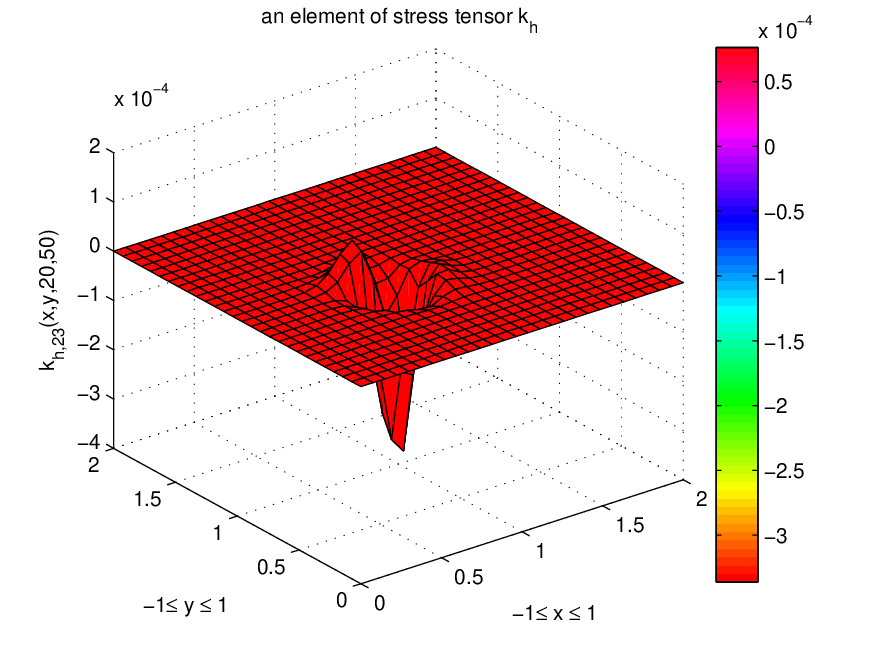,width=5.5cm}
         \end{tabular}
        \end{center}
         \caption{Graphs of displacement and stress tensor corresponding to Mbankolo landslides (first four figures) and Gouache landslides (last four figures).}
          \label{fig4}
          \end{figure}

           Under the time step restriction $(\ref{sr})$, both displacement and elements of the symmetric stress tensor provided by the developed modified Lax-Wendroff/interpolation approach with finite element method are displayed in Figure $\ref{fig2}$. This figure suggests that the new computational technique $(\ref{28})$-$(\ref{31})$ for solving the three-dimensional system of tectonic deformation equation $(\ref{1})$ subjects to initial-boundary conditions $(\ref{2})$-$(\ref{3})$ is stable and convergent under the suitable condition $(\ref{sr})$. It's important to remind that the error represented in Figure $\ref{fig2}$ is associated with the displacement. Additionally, it follows from the graphs related to displacement and stress tensor that the geological structure deformation starts at a focus centered in the domain $\Omega$ and the debris propagate and reach a maximum value. Figure $\ref{fig3}$ describes the behavior of Dschang cliff landslides during the period $5$-$7$ November $2024$ using a slope action equals $7.6\%$ and an altitude difference of $0.7km$. The first four graphs in the figure show the behavior of displacement and stress tensor elements (first graph) along with the first component of displacement and two elements of the stress tensor obtained using a slope action $7.6\%$ while the last four graphs present both displacement and stress tensor coefficients together with the second component of displacement and two elements of the tensor utilizing an altitude difference equals $0.7km$. We observe from Figure $\ref{fig3}$ that the graphs corresponding to a slope action $7.6\%$ are almost similar to the ones associated with an altitude difference $0.7km$. This suggests that the landslides hazards should be assessed and predicted by the new algorithm $(\ref{28})$-$(\ref{31})$ using either a slope action or an altitude difference. In addition, the use of other parameters such as: temperature, rainfall intensity, etc..., in the developed computational technique could also provide useful information regarding the disaster risks. Finally, the displacement and symmetric stress tensor corresponding to Mbankolo and Gouache landslides are displayed in Figure $\ref{fig4}$ with associated difference altitudes: $0.0438km$ and $0.209km$, respectively. As already indicated in Dschang cliff events, the first four graphs consider the Mbankolo damages occurred from $6$-$8$ October $2023$, whereas the last four ones deal with the Gouache disasters occurred on $29$ October $2019$.\\

       \section{General conclusions and future works}\label{sec5}
        In this paper, we have developed a modified Lax-Wendroff/interpolation approach with finite element procedure in an approximate solution of a three-dimensional system of geological structure deformation model $(\ref{1})$, subjects to initial and boundary conditions $(\ref{2})$ and $(\ref{3})$, respectively. The stability and error estimates of the new computational technique $(\ref{28})$-$(\ref{31})$ have been deeply analyzed using the $L^{\infty}(0,T_{f};\text{\,}L^{2})$-norm. Under an appropriate time step limitation $(\ref{sr})$, the theoretical studies suggested that the new algorithm is stable, temporal second-order accurate and spatial convergent with order $O(h^{p})$, where $h$ denotes the grid space and $p$ is a positive integer greater than or equal $2$. A wide set of numerical experiments have confirmed the theory. Specifically, the graphs (Figures $\ref{fig2}$) show that the proposed approach $(\ref{28})$-$(\ref{31})$ is stable whereas \textbf{Tables 1-2} indicate that the developed computational method is second-order convergent in time and spatial third-order accurate. Furthermore, both Figures $\ref{fig2}$ and Tables $1$-$2$ show that the displacement and stress tensor elements do not increase with time. This observation suggests that the proposed approach overcomes the drawbacks raised by a wide set of numerical methods widely discussed in the literature \cite{4jl,12jl,6jl,7jl,lyl}. Finally, an example of a three-dimensional system of elastodynamic problem considers the landslide disasters observed in the west and center regions of Cameroon from $29$ October $2019$ up to $5$-$7$ November $2024$. Specifically, some discussions on the landslides displacement and stress tensor associated with Dschang cliff disasters ($5$-$7$ November $2024$), Mbankolo landslides ($6$-$8$ October $2023$) and Gouache damages ($29$ October $2019$) are presented in Tables $3$-$6$ and Figures $\ref{fig3}$-$\ref{fig4}$ utilizing either a slope action or an altitude difference. These tables and figures suggest that the new algorithm $(\ref{28})$-$(\ref{31})$ should be considered as a powerful tool for providing useful data and related information on some natural disasters in Cameroon (also in the world) which would help communities to be informed about the hazard zones and allow people to avoid the landslide damages. The development of a combined Lax-Wendroff/Crank-Nicolson technique with finite element method for solving a three-dimensional system of nonlinear tectonic deformation problem which is more suitable to model large landslides, rocks avalanche and other natural disasters will be the subject of our future works.

      \subsection*{Ethical Approval}
     Not applicable.
     \subsection*{Availability of supporting data}
     Not applicable.
     \subsection*{Declaration of Interest Statement}
     The author declares that he has no conflict of interests.
     \subsection*{Funding}
     Not applicable.
     \subsection*{Authors' contributions}
     The whole work has been carried out by the author.

     \end{document}